\documentclass[ 11pt]{article}

\usepackage[utf8]{inputenc}  
\usepackage[T1]{fontenc}         
\usepackage{geometry}         
\usepackage[english]{babel}  
\usepackage{cite}
\usepackage{amssymb}   
\usepackage{caption}
\usepackage{amsmath} 
\usepackage{amscd}
\usepackage{graphicx, color}    
\usepackage{amsthm}                 
\usepackage{latexsym}     
\usepackage[all]{xy}
\usepackage{float}

\newtheorem{thm}{Theorem}[section]
\newtheorem{cor}[thm]{Corollary}
\newtheorem{prop}[thm]{Proposition}
\newtheorem{lem}[thm]{Lemma}

\theoremstyle{definition}
\newtheorem{definition}[thm]{Definition}

\newtheorem{rmk}[thm]{Remark}
\def\lquotient#1#2{%
\makeatletter
\lower.3ex\hbox{$#1$}\backslash\raise.3ex\hbox{$#2$}%
\makeatother
}															

\def\rquotient#1#2{%
\makeatletter
\raise.3ex\hbox{$#1$}/\lower.3ex\hbox{$#2$}%
\makeatother
}

\newcommand{\cC}{{\mathcal C}}

\newcommand{\cR}{{\mathcal R}}

\newcommand{\cY}{{\mathcal Y}}

\newcommand{\ra}{\rightarrow}

\newcommand{\hra}{\hookrightarrow}

\title{\textbf{Complexes of groups and geometric small cancellation over graphs of groups}}           
\author{Alexandre Martin\footnote{This research is supported by the European Research Council (ERC) grant of G. Arzhantseva, grant agreement n$^o$ 259527.}}
\date{}

\begin{document}
\maketitle

\noindent {\bf Abstract.} We explain and generalise a construction due to Gromov to realise geometric small cancellation groups over graphs of groups as fundamental groups of non-positively curved 2-dimensional complexes of groups. We then give conditions so that the hyperbolicity and some finiteness properties of the small cancellation quotient can be deduced from analogous properties for the local groups of the initial graph of groups. \\

\noindent {\bf Primary MSC.}  20F65. \\
\noindent {\bf Secondary MSC}. 20E08, 57M07.\\

\noindent Alexandre Martin\\
\noindent Fakult\"at f\"ur Mathematik, Universit\"at Wien\\
\noindent   Oskar-Morgenstern-Platz 1, 1180 Wien, Austria.

\noindent E-mail: alexandre.martin@univie.ac.at\\

\section{Introduction and statement of results.}

Small cancellation theory deals with the following problem: given  a group $G$ and a family $(H_i)_{i \in I}$ of subgroups, find conditions under which one understands the quotient $\rquotient{G}{\ll H_i\gg }$ (where $\ll H_i\gg $ denotes the normal closure of the subgroup generated by the $H_i$). 

In classical small cancellation theory, $G$ is a finitely generated free group and each $H_i$ is an infinite cyclic subgroup generated by a cyclically reduced element. Small cancellation conditions essentially ask that the length of a common subword of two relators be short relatively to the length of the relators. Such conditions come in various flavours and the overall theory has been generalised to many settings, such as small cancellation over graphs of groups \cite{LyndonSchupp} or small cancellation over a hyperbolic group \cite{OlshanskiiResidualingHomomorphisms, DelzantSousGroupesDistingues, ChampetierPetiteSimplificationHyperbolique}. In \cite{GromovRandomWalkRandomGroups}, Gromov gave a geometric version of small cancellation, using the language of \textit{rotation families} (see Definition \ref{rotationfamily}). In this case, the group $G$ acts isometrically on a hyperbolic metric space $X$ and each subgroup $H_i$ stabilises a given subspace $Y_i$. Small cancellation conditions in this context ask that the overlap between two such subspaces be small with respect to the injectivity radii of the $H_i$. This point of view was used for instance in \cite{CoulonAsphericity, DelzantGromovCourbureMesoscopique, DahmaniGuirardelOsin}.

Small cancellation theory offers powerful tools to study various classes of groups, and provide many examples of groups with exotic properties \cite{GromovRandomWalkRandomGroups}. \\ 

Some small cancellation conditions have strong geometric consequences. In \cite{GromovCATkappa}, Gromov proved that groups satisfying the geometric small cancellation condition $C''(1/6)$ act properly and cocompactly on a CAT(0) space. The first goal of this article is to detail this construction and to extend it to the case of small cancellation over a graph of groups; in the case of classical small cancellation, this construction was explained by Vinet (unpublished). More precisely, we prove the following: 

\begin{thm}
 Let $G(\Gamma)$ be a graph of groups over a finite simplicial graph $\Gamma$, with fundamental group $G$ and Bass--Serre tree $T$. Let $(A_\xi, H_\xi)_{\xi \in \Xi}$ a rotation family such that:
\begin{itemize}
\item[{\normalfont \mbox{(RF$_1$)}}] each $H_\xi$ is an infinite cyclic subgroup generated by a hyperbolic element with axis $A_\xi$,
\item[{\normalfont \mbox{(RF$_2$)}}] for every $ \xi \in\Xi$, the set of elements  $g \in G$ such that $gA_\xi= A_\xi$ is infinite cyclic. 
\end{itemize}
If $(A_\xi, H_\xi)_{\xi \in \Xi}$ satisfies the geometric small cancellation condition $C''(1/6)$, then the quotient $~\rquotient{G}{\ll H_\xi \gg }$ is the fundamental group of a non-positively curved 2-dimensional complex of groups, the local groups of which are either finite or subgroups of the local groups of $G(\Gamma)$.
\label{mainPS1} 
\end{thm}

In \cite{GromovCATkappa}, Gromov even constructed actions of $C''(1/6)$ small cancellation groups on CAT($\kappa$) spaces for some $\kappa < 0$. While it would be possible to adapt the previous constructions to obtain such actions, we restrict to actions on (piecewise Euclidean) CAT(0) complexes for two reasons. First, the actions considered here are non-proper, so a negatively-curved assumption on the space would not translate immediately to a property of the quotient group (note however that we will consider the hyperbolicity of such a quotient in Section \ref{hyperbolicityquotient}). More importantly, having an action on a piecewise-Euclidean complex (or more generally on a $ M_\kappa$-complex in the sense of Bridson \cite{BridsonPhD}) will be used in Section \ref{hyperbolicityquotient} to study the geometry of the quotient group.

Let us give a few details about the construction. The family of axes $A_\xi$ is used to construct the so-called coned-off space $\widehat{T}$ (see Definition \ref{coneoff}) and the quotient  space $\rquotient{\widehat{T}}{\ll H_\xi \gg }$ comes with an action of $\rquotient{G}{\ll H_\xi\gg }$. Since two axes $A_\xi, A_{\xi'}$ may share more than one edge, the space $\rquotient{\widehat{T}}{\ll H_\xi \gg }$ does not have a CAT(0) geometry in general. Generalising an idea of Gromov \cite{GromovCATkappa}, we want to construct a CAT(0) complex with a cocompact action of $\rquotient{G}{\ll H_\xi\gg }$ by identifying certain portions of $\rquotient{\widehat{T}}{\ll H_\xi \gg }$. Understanding the resulting action, and in particular the various stabilisers, turns out to be a non-trivial problem. To avoid this issue, we use a different approach, using the theory of complexes of group. We start by considering the complex of groups one would except from the action of $\rquotient{G}{\ll H_\xi\gg }$ on the space obtained from $\rquotient{\widehat{T}}{\ll H_\xi \gg }$ after performing Gromov's construction. This complex of groups decomposes as a tree of complexes of groups (see Section \ref{trees} for the definition), where the various pieces are much easier to handle. In particular, our approach allows us to consider only the action of $G$ on $\widehat{T}$, instead of dealing with the action of $\rquotient{G}{\ll H_\xi\gg }$ on $\rquotient{\widehat{T}}{\ll H_\xi \gg }$. Using standard results on complexes of groups, we prove that this complex of groups is indeed non-positively curved and has $\rquotient{G}{\ll H_\xi\gg }$ as fundamental group. Note that this approach provides a purely geometric proof of two facts which are well known for classical small cancellation theory: 

\begin{prop}
Under the assumptions of Theorem \ref{mainPS1}, the quotient map $G \ra \rquotient{G}{\ll H_\xi \gg }$ embeds each local group of $G$. 
\end{prop}

\begin{prop}
Under the assumptions of Theorem \ref{mainPS1}, every torsion element $g$ of $\rquotient{G}{\ll H_\xi\gg }$ satisfies either 
\begin{itemize}
\item $g$ is conjugate to the projection of a torsion element in a local factor of $G$, or
\item $g$ is conjugate to the projection of an element of $G$ a power of which is in $\ll H_\xi\gg$. 
\end{itemize} 
\end{prop} 

In Theorem \ref{mainPS1}, we make the assumption  (RF$_2$) that the (global) stabiliser of any axis $A_\xi$ is infinite cyclic.  This condition is required to adapt the construction of  Gromov to this more general setting. Note that it is automatically satisfied for classical small cancellation theory, that is, for the action of the free group $F_n$ on the $2n$-valent tree. Under condition (RF$_1$), condition (RF$_2$) is equivalent to requiring that: 
\begin{itemize}
\item no non-trivial element of $G$ pointwise fixes an axis $A_\xi$,
\item no element of $G$ reflects an axis $A_\xi$ across a point.
\end{itemize}
This is for instance satisfied for a group $G$ without $2$-torsion acting acylindrically on a simplicial tree, where we say that an action is \textit{acylindrical} if  there exists a uniform upper bound on the distance between two points which are fixed by an infinite subgroup of $G$. \\

In a nutshell, the previous theorem asserts that some small cancellation quotients act in a  well controlled way on a CAT(0) space. With such an action at hand, a natural problem is to determine which properties of the group can be deduced from analogous properties of the stabilisers of simplices. In \cite{MartinBoundaries, MartinCombinationEG}, the author studied such combination problems in the case of cocompact actions on CAT(0) simplicial complexes. In particular, the following acylindrical generalisation of the Bestvina--Feighn combination theorem for hyperbolic groups \cite{BestvinaFeighnCombinationHyperbolic} was proved:

\begin{thm}[M.\cite{MartinBoundaries, MartinCombinationEG}]
Let $G(\cY)$ be a non-positively curved complex of groups over a finite piecewise-Euclidean complex $Y$, with fundamental group $G$ and universal cover $X$. Assume that:
\begin{itemize}
\item The universal cover $X$ is hyperbolic (for the associated piecewise-Euclidean structure),
\item The local groups are hyperbolic and all the local maps are quasiconvex embeddings,
\item The action of $G$ on $X$ is acylindrical. 
\end{itemize}
Then $G$ is hyperbolic and the local groups embed in $G$ as quasiconvex subgroups. \qed
\label{combinationhyperbolic}
\end{thm}

We use this result to prove the following combination theorem for hyperbolic groups:

\begin{thm}
There exists a universal constant $0<\lambda_\mathrm{univ} \leq \frac{1}{6}$ such that the following holds. Let $G(\Gamma)$ be a graph of groups over a finite simplicial graph $\Gamma$, with fundamental group $G$ and Bass--Serre tree $T$.  Let $(A_\xi, H_\xi)_{\xi \in \Xi}$ a rotation family satisfying conditions {\normalfont (RF$_1$)} and {\normalfont (RF$_2$)}. Suppose in addition that the following holds:
\begin{itemize}
\item[{\normalfont (HC$_1$) }] The local groups of $G(\Gamma)$ are hyperbolic and all the local maps are quasiconvex embeddings,
\item[{\normalfont (HC$_2$) }] The action of $G$ on $T$ is acylindrical, 
\item[{\normalfont (HC$_3$) }] There are only finitely many elements in $\Xi$ modulo the action of $G$. 
\end{itemize}
If $(A_\xi, H_\xi)_{\xi \in \Xi}$ satisfies the geometric small cancellation condition $C''(\lambda_\mathrm{univ})$, then the quotient group $\rquotient{G}{\ll H_\xi \gg }$ is hyperbolic and the projection $G \ra \rquotient{G}{\ll H_\xi \gg }$ embeds each local group of $G(\Gamma)$ as a quasiconvex subgroup.
\label{mainPS2}
\end{thm}

In Theorem \ref{mainPS2}, the group $G$ itself is hyperbolic as a consequence of the Bestvina--Feighn combination theorem \cite{BestvinaFeighnCombinationHyperbolic}. While small cancellation theory over a hyperbolic group  studies quotients of hyperbolic groups by means of their action on their (hyperbolic) Cayley graphs, the previous theorem provides small cancellation tools to study quotients of hyperbolic groups that split by means of their actions on their associated Bass--Serre trees.

Note that the geometry of the quotient complexes and quotient groups under the action of a rotation family has been considered for instance in \cite{CoulonAsphericity, DahmaniGuirardelOsin}. This article however is to the author's knowledge one of the first that proves the hyperbolicity of a small cancellation quotient in the case of non-proper actions.\\

In particular, Theorem \ref{mainPS2} has the following corollaries. Recall that, given a group $G$, a family of subgroups $(H_i)_{i \in I}$ is said to be \textit{almost malnormal} in $G$ if the following holds: 
\begin{itemize}
\item for every $i \neq j$ in $I$ and every $g,h \in G$, $gH_ig^{-1} \cap hH_jh^{-1}$ is finite, 
\item for every $i  \in I$ and every $g \in G - H_i$, $gH_ig^{-1} \cap H_i$ is finite.
\end{itemize}

\begin{cor}
Let $G=G_1 *_M G_2 $ be an amalgamated product such that $G_1$ and $G_2$ are hyperbolic groups without $2$-torsion, and $M$ embeds in both $G_1$ and $G_2$ as an almost malnormal quasiconvex subgroup. Let $(A_\xi, H_\xi)_{\xi \in \Xi}$ be a rotation family as in Theorem \ref{mainPS2} and satisfying the small cancellation $C''(\lambda_\mathrm{univ})$. Then the quotient group $\rquotient{G}{\ll H_\xi \gg }$ is hyperbolic and the projection $G \ra \rquotient{G}{\ll H_\xi \gg }$ embeds $G_1$ and $G_2$ as quasiconvex subgroups.\qed
\end{cor}

\begin{cor}
Let $G= G_1 *_\varphi$ be an HNN extension associated to an isomorphism $\varphi: H_1 \ra H_2$ between subgroups $H_1, H_2$ of $G_1$, such that $G_1$ is a hyperbolic group without $2$-torsion, $H_1$ and $H_2$ are quasiconvex in $G_1$ and the family $(H_1, H_2)$ is almost malnormal in $G_1$. Let $(A_\xi, H_\xi)_{\xi \in \Xi}$ be a rotation family as in Theorem \ref{mainPS2} and satisfying the small cancellation $C''(\lambda_\mathrm{univ})$. Then the quotient group $\rquotient{G}{\ll H_\xi \gg }$ is hyperbolic and the projection $G \ra \rquotient{G}{\ll H_\xi \gg }$ embeds $G_1$ as a quasiconvex subgroup.\qed
\end{cor}

By adapting the previous construction, we obtain similar combination results for finiteness properties. In order to do that, we use the following combination result, which generalises a construction due to Haefliger \cite{HaefligerExtension}.

\begin{thm}[M. \cite{MartinCombinationEG}]
\label{combinationEG}
Let $G(\cY)$ be a developable complex of groups over a finite simplicial complex $Y$, with fundamental group $G$ and universal cover $X$. Suppose that:
\begin{itemize}
\item for every finite $H$ subgroup of $G$, the  fixed-point set $X^H$ is contractible,
\item every local group admits a cocompact model of classifying space for proper actions.
\end{itemize}  
Then there exists a cocompact model of classifying space for proper actions for $G$.\qed
\end{thm}

In particular, we use this to prove the following:

\begin{thm}
Let $G(\Gamma)$ be a graph of groups over a finite simplicial graph $\Gamma$, with fundamental group $G$ and Bass--Serre tree $T$.  Let $(A_\xi, H_\xi)_{\xi \in \Xi}$ a rotation family that satisfies conditions {\normalfont (RF$_1$)} and {\normalfont (RF$_2$)} of Theorem \ref{mainPS1}, as well as the geometric small cancellation condition $C''(1/6)$. If all the local groups of $G(\Gamma)$ admit cocompact models of classifying spaces for proper actions, then so does $\rquotient{G}{\ll H_\xi \gg }$.
\label{mainPS3}
\end{thm}

Here is an outline of the article. Section \ref{complexesofgroups} contains some background and notations on complexes of groups. Section \ref{trees} contains gluing constructions for complexes of groups which are reminiscent of the theory of orbispaces introduced by Haefliger \cite{HaefligerOrbihedra}. Section \ref{smallcancellation} gives a short presentation of geometric small cancellation theory over a graph of groups, using the language of rotation families. Given a small cancellation group $\rquotient{G}{\ll H_\xi \gg }$ over a graph of groups as in Theorem \ref{mainPS1}, we construct in Section \ref{Gromovconstruction} a non-positively curved $2$-dimensional complex of groups that admit $\rquotient{G}{\ll H_\xi \gg }$ as its fundamental group, by generalising a construction of Gromov. Section \ref{EG} deals with finiteness properties and explains the construction of a cocompact model of classifying space for proper actions for $ \rquotient{G}{\ll H_\xi  \gg }$. We study the action resulting from our construction in Section \ref{hyperbolicityquotient} and prove the hyperbolicity of the quotient group. Finally, Section \ref{CAT0} is a technical section proving that Gromov's construction does indeed produce a CAT(0) space. \\

\noindent \textbf{Acknowledgements.} I would like to thank Thomas Delzant for proposing this problem to me and for many useful discussions during this work.

\section{Background on complexes of groups.}
\label{complexesofgroups}

\subsection{First definitions.}

Graphs of groups are algebraic objects that were introduced by Serre \cite{SerreTrees} to encode group actions on simplicial trees. If one wants to generalise that theory to higher dimensional complexes, one needs the theory of complexes of groups developed by Gersten-Stallings \cite{GerstenStallings}, Corson \cite{CorsonComplexesofGroups} and Haefliger \cite{HaefligerOrbihedra}. Haefliger defined a notion of complexes of groups over more general objects called \textit{small categories without loops} (abbreviated \textit{scwol}), a combinatorial generalisation of polyhedral complexes. Although in this article we will only deal with actions on simplicial complexes, we use the terminology of scwols to be coherent with the existing literature on complexes of groups. For a deeper treatment of the material covered in this paragraph and for the general theory of complexes of groups over scwols, we refer the reader to \cite{BridsonHaefliger}.

\begin{definition}[small category without loop]
A  \textit{small category without loop} (briefly a \textit{scwol}) is a set $\cY$ which is the disjoint union of a set $V(\cY)$ called the vertex set of $\cY$, and a set $E(\cY)$ called the edge set of $\cY$, together with maps
$$i: E(\cY)  \ra V(\cY) \mbox{ and } t: E(\cY) \ra V(\cY).$$
For an edge $a \in E(\cY)$, $i(a)$ is called the initial vertex of $a$ and $t(a)$ the terminal vertex of $a$. 

Let $E^{(2)}(\cY)$ be the set of pairs $(a, b) \in E(\cY)$ such that $i(a) = t(b)$. A third map 
$$ E^{(2)}(\cY) \ra E(\cY) $$
is given that associates to such a pair $(a,b)$ an edge $ab$ called their composition (and $a$ and $b$ are said to be composable). These maps are required to satisfy the following conditions:
\begin{itemize}
 \item For every $(a,b) \in E^{(2)}(\cY)$, we have $i(ab)=i(b)$ and $t(ab)=t(a)$; 
 \item For every $a,b,c \in E(\cY)$ such that $i(a) = t(b)$ and $i(b)=t(c)$, we have $(ab)c = a(bc)$ (and the composition is simply denoted $abc$).
 \item For every $a \in E(\cY)$, we have $t(a) \neq i(a)$.
\end{itemize}
\end{definition}

\begin{definition}[simplicial scwol associated to a simplicial complex] If $Y$ is a simplicial complex, a scwol $\cY$ is naturally associated to $Y$ in the following way: 
\begin{itemize}
 \item $V(\cY)$ is the set $S(Y)$ of simplices of $Y$, 
 \item $E(\cY)$ is the set of pairs $(\sigma, \sigma') \in V(\cY)^2$ such that $\sigma \subset \sigma'$.
 \item For a pair $a=(\sigma, \sigma') \in E(\cY)$, we set $i(a) = \sigma'$ and $t(a)=\sigma$.
 \item For composable edges $a=(\sigma, \sigma')$ and $b=(\sigma', \sigma'')$, we set $ab=(\sigma, \sigma'')$.
\end{itemize}
We call $\cY$ the \textit{simplicial scwol associated to $Y$}.
\end{definition}

In what follows, we will often omit the distinction between a simplex $\sigma$ of $Y$ and the associated vertex of $\cY$.

\begin{definition}[Complex of groups \cite{BridsonHaefliger}]
 Let $\cY$ be a scwol. A \textit{complex of groups $G(\cY)= (G_\sigma, \psi_a, g_{a,b})$ over $\cY$} is given by the following data: 
\begin{itemize}
 \item for each vertex $\sigma$ of $\cY$, a group $G_\sigma$ called the \textit{local group} at $\sigma$,
 \item for each edge $a$ of $\cY$, an injective homomorphism $\psi_a: G_{i(a)} \ra G_{t(a)}$,
 \item for each pair of composable edges $(a,b)$ of $\cY$, a \textit{twisting element} $g_{a,b} \in G_{t(a)}$, with the following compatibility conditions:
\begin{enumerate}
 \item for every pair $(a,b)$ of composable edges of $\cY$, we have 
$$ \mbox{Ad}(g_{a,b}) \psi_{ab} =  \psi_a \psi_b,$$
where $\mbox{Ad}(g_{a,b})$ is the conjugation by $g_{a,b}$ in $G_{t(a)}$; 
 \item if $(a,b)$ and $(b,c)$ are pairs of composable edges of $\cY$, then the following cocycle condition holds:
$$ \psi_a(g_{b,c})g_{a,bc}= g_{a,b}g_{ab,c}.$$
\end{enumerate}
\end{itemize}
If $Y$ is a simplicial complex, a complex of groups \textit{over $Y$} is a complex of groups over the associated simplicial scwol $\cY$.
\end{definition}

\begin{definition}[Morphism of complex of groups]
 Let $Y, Y'$ be simplicial complexes, $\cY$ (resp. $\cY'$) the associated simplicial scwols, $f: Y \ra Y'$ a non-degenerate simplicial map (that is, the restriction of $f$ to any simplex is a homeomorphism onto its image), and $G(\cY)$ (resp. $G(\cY')$) a complex of groups over $Y$ (resp. $Y'$). A \textit{morphism} $F=(F_\sigma, F(a))$ : $G(\cY) \ra G(\cY')$ over $f$ consists of the following:
\begin{itemize}
 \item for each vertex $\sigma$ of $\cY$, a homomorphism $F_{\sigma}: G_\sigma \ra G_{f(\sigma)}$, 
 \item for each edge $a$ of $\cY$, an element $F(a) \in G_{t(f(a))}$ such that
   \begin{enumerate}
    \item for every edge $a$ of $\cY$, we have
$$ \mbox{Ad}(F(a))\psi_{f(a)}F_{i(a)} = F_{t(a)}\psi_a,$$
    \item for every pair $(a,b)$ of composable edges of $\cY$, we have
$$ F_{t(a)}(g_{a,b})F(ab) = F(a)\psi_{f(a)}(F(b))g_{f(a),f(b)}.$$
   \end{enumerate}
\end{itemize}
If all the $F_\sigma$ are isomorphisms, $F$ is called a \textit{local isomorphism}. If in addition $f$ is a simplicial isomorphism, $F$ is called an \textit{isomorphism}.

A \textit{morphism from $G(\cY)$ to a group $G$} is a morphism of complexes of groups from $G(\cY)$ to the scwol with a unique vertex and no edge, and with $G$ as its (unique) local group. Equivalently, it consists of a homomorphism $F_\sigma: G_\sigma \ra G$ for every $\sigma \in V(\cY)$ and an element $F(a) \in G$ for each $a \in E(\cY)$ such that:
\begin{itemize}
 \item for every $a \in E(\cY)$, we have $F_{t(a)} \psi_a = \mbox{Ad}(F(a)) F_{i(a)}$,
 \item for every pair $(a,b)$ of composable edges of $\cY$, we have $F_{t(a)}(g_{a,b})F(ab)=F(a)F(b)$.
\end{itemize}
\end{definition}

\subsection{Developability and non-positively curved complexes of groups.}

\begin{definition}[Complex of groups associated to an action without inversion of a group on a simplicial complex \cite{BridsonHaefliger}]
 Let $G$ be a group acting without inversion by simplicial isomorphisms on a simplicial complex $X$, let $Y$ be the quotient space and $p:X \ra Y$ the natural projection. Up to a barycentric subdivision, we can assume that $p$ restricts to an embedding on every simplex, yielding a simplicial structure on $Y$. Let $\cY$ be the simplicial scwol associated to $Y$.

For each vertex $\sigma$ of $\cY$, choose a simplex $\widetilde{\sigma}$ of $X$ such that $p(\widetilde{\sigma})=\sigma$. As $G$ acts without inversion on $X$, the restriction of $p$ to any simplex of $X$ is a homeomorphism onto its image. Thus, to every simplex $\sigma'$ of $Y$ contained in $\sigma$, there is a unique $\tau$ of $X$ and contained in $\widetilde{\sigma}$, such that $p(\tau) = \sigma'$. To the edge $a= (\sigma', \sigma)$ of $\cY$ we then choose an element $h_a \in G$ such that $h_a.\tau = \widetilde{\sigma}'$. A \textit{complex of groups $G(\cY)=(G_\sigma, \psi_a, g_{a,b})$ over $Y$ associated to the action of $G$ on $X$} is given by the following:
\begin{itemize}
 \item for each vertex $\sigma$ of $\cY$, let $G_\sigma$ be the stabiliser of $\widetilde{\sigma}$, 
 \item for every edge $a$ of $\cY$, the homomorphism $ \psi_a: G_{i(a)} \ra G_{t(a)}$ is defined by 
$$ \psi_a(g) = h_agh_a^{-1},$$
 \item for every pair $(a,b)$ of composable edges of $\cY$, define
$$g_{a,b} = h_ah_bh_{ab}^{-1}.$$
\end{itemize}
Moreover, there is an associated morphism $F=(F_\sigma, F(a))$ from $G(\cY)$ to $G$, where $F_\sigma: G_\sigma \ra G$ is the natural inclusion and $F(a) = h_a$.
\label{inducedcomplexofgroups}
\end{definition}

\begin{definition}[Developable complex of groups]
 A complex of groups over a simplicial complex $Y$ is \textit{developable} if it is isomorphic to the complex of groups associated to an action without inversion on a simplicial complex.
\end{definition}
 
Unlike in Bass--Serre theory, not every complex of groups is developable. Checking whether or not a complex of groups is developable is a non trivial problem in general. We have the following characterisation of developability: 

\begin{thm}[Theorem III.$\cC$.2.13, Corollary III.$\cC$.2.15, Theorem III.$\cC$.3.13 and Corollary III.$\cC$.3.15 of \cite{BridsonHaefliger}] Let $G(\cY) = (G_\sigma, \psi_a, g_{a,b})$ be a complex of groups over a simplicial complex $Y$. 
\begin{itemize}
\item $G(\cY)$ is developable if and only if there exists a group $G$ and a morphism  $F:G(\cY) \ra G$ that is injective on the local groups. 
\item To each such morphism $F: G(\cY) \ra G$ that is injective on the local groups, one can associate an action of $G$ on a simplicial complex, called the \textit{development} associated to $F$, such that the induced complex of groups is isomorphic to $G(\cY)$. 
\item If $G(\cY)$ is developable, then there exists a group $G$ and a morphism $F: G(\cY) \ra G$ which is injective on the local groups and such that the associated development is connected and simply-connected. Such a group, which is unique up to isomorphism, is called the \textit{fundamental group} of $G(\cY)$. The development associated to such a morphism, which is unique up to equivariant isomorphism, is called the \textit{universal cover} of $G(\cY)$. \qed
\end{itemize}
\label{developabilityalgebraic}
\end{thm}

There exists an explicit description of the fundamental group of a complex of groups which generalises the usual definition of the fundamental group of a space in terms of loops. This construction being rather technical, we omit it and refer to \cite[Definition III.$\cC$.3.5]{BridsonHaefliger} for details. We will just use the following functoriality property: 

\begin{prop}
Let $G(\cY)$ be a developable complex of groups over a simplicial complex $Y$ and $v$ be a vertex of $Y$. Then there exists a group $\pi_1(G(\cY), v)$ together with a morphism $F_{G(\cY),v}: G(\cY) \ra \pi_1(G(\cY), v)$ which is injective on the local groups, such that the associated development is connected and simply connected (i.e. $\pi_1(G(\cY), v)$ is a fundamental group of $G(\cY)$). 

Moreover, if $G(\cY), G(\cY')$ are developable complexes of groups over simplicial complexes $Y, Y'$, $F: G(\cY) \ra G(\cY')$ is a morphism of complexes of groups over a simplicial map $f:Y \ra Y'$ and $v$ is a vertex of $Y$, then there is a morphism $F_*: \pi_1(G(\cY), v) \ra \pi_1(G(\cY'), f(v))$ such that for every simplex $\sigma$ of $Y$, the following diagram commutes: 
$$ \xymatrix{
     \pi_1(G(\cY), v)  \ar[r]_{}^{F_*} & \pi_1(G(\cY'), f(v)) \\
    G_\sigma  \ar[u]_{(F_{G(\cY),v})_\sigma} \ar[r]_{F_\sigma} & G_{f(\sigma)} \ar[u]_{(F_{G(\cY'),f(v)})_{f(\sigma)}} . \\
  }$$\\ \qed
\end{prop}
We now recall a geometric condition that ensures the developability of a given complex of groups. From now on, we assume that $Y$ is endowed with a piecewise-Euclidean structure. 

\begin{definition}[Local complex of groups] Let $v$ be a vertex of $Y$. We denote by $G(\cY_v)$ the complex of groups over the star $\mbox{St}(v)$ of $v$ induced by $G(\cY)$ in the obvious way.
\end{definition}

We have the following result:

\begin{prop}[Proposition III.$\cC$. 4.11 of \cite{BridsonHaefliger}] For every vertex $v$ of $Y$, the local complex of groups $G(\cY_v)$ is developable and its fundamental group is isomorphic to $G_v$. 
 Denote by $X_v$ its universal cover, called the \textit{local development} at $v$. Then the piecewise-Euclidean structure on the star $\mbox{St}(v)$ yields a piecewise-Euclidean structure with finitely many isometry types of simplices on $X_v$ such that the $G_v$-equivariant projection $X_v \ra \mbox{St}(v)$ restricts to an isometry on every simplex.\qed
\end{prop}

\begin{definition}[non-positively curved complex of groups] We say that $G(\cY)$ is \textit{non-positively curved} if each local development $X_v$ with the simplicial metric coming from the piecewise-Euclidean structure of $Y$ is a CAT(0) space.
\end{definition}

\begin{thm}[Theorem III.$\cC$.4.17 of \cite{BridsonHaefliger}]
A non-positively curved complex of groups is developable. \qed
\label{nonpositivelycurveddevelopable}
\end{thm}

\section{Trees of complexes of groups.}
\label{trees}

This section presents the main algebraic construction of this article. Given a simplicial complex $Y$, subcomplexes $(Y_i)$ whose interiors cover $Y$ and such that the nerve of the associated open cover is a tree, and a family of complexes of groups $G(\cY_i)$ over $Y_i$, we explain how one can glue them together to obtain a complex of groups $G(\cY)$ over $Y$. This procedure can be thought as making ``trees of complexes of groups''. In order to lighten notations, we will only detail the case of a cover consisting of two subcomplexes with a connected intersection.

\subsection{Immersions of complexes of groups.}

\begin{definition}[Immersion of complexes of groups.] Let $G(\cY_1)$ and $G(\cY_2)$ be two complexes of groups over two simplicial complexes $Y_1$ and $Y_2$ and $F: G(\cY_1) \ra G(\cY_2)$ a morphism of complexes of groups over a simplicial morphism $f: Y_1 \ra Y_2$. We say that $F$ is an \textit{immersion} if $f$ is a simplicial immersion and all the local maps $F_\sigma$ are injective.
\end{definition}

Note that if in addition both complexes of groups are assumed to be developable, then the simplicial immersion $f$ lifts to an equivariant simplicial immersion between their universal coverings. \\

For $i=1, 2$, let $X_i, Y_i$ be simplicial complexes, $G_i$ a group acting without inversion on $X_i$ by simplicial isomorphisms, $p_i: X_i \ra Y_i$ a simplicial map factoring through $X_i/G_i$ and inducing a simplicial isomorphism $X_i/G_i \simeq Y_i$. 
Suppose we are given a simplicial immersion $f: Y_1 \ra  Y_2$, a homomorphism $\alpha: G_1 \ra G_2$ and an equivariant simplicial immersion $\overline{f}: X_1 \ra X_2$ over $f$ such that for every simplex $\sigma$ of $X_1$, the induced map $\alpha: \mbox{Stab}(\sigma) \ra \mbox{Stab}(\bar{f}(\sigma))$ is a monomorphism. Recall that from the action of $G_i$ on $X_i$, it is possible to define a complex of groups over $Y_i$ that encodes it. We now explain how to define such complexes of groups $G(\cY_1)$ and $G(\cY_2)$ over $Y_1$ and $Y_2$, such that there is an immersion $G(\cY_1) \ra G(\cY_2)$. 

Recall that to define a complex of groups over $Y_1$ induced by the action of $G_1$ on $X_1$, we had to associate to every vertex $\sigma$ of $\cY_1$ a simplex $\overline{\sigma}$ of $X_1$, and to every edge $a$ of $\cY_1$ an element $h_a$ of $G_1$ (see Definition \ref{inducedcomplexofgroups}). Assume we have made such choices to define $G(\cY_1)$. We now make such choices for $Y_2$.

\begin{itemize}
\item Let $\sigma'$ be a vertex of $\cY_2$, which we identify with the associated simplex of $Y_2$. If $\sigma' = f(\sigma)$ for a simplex $\sigma$ of $Y_1$, we choose $\overline{\sigma'} = \overline{f}(\overline{\sigma})$. Otherwise, we pick an arbitrary lift of $\sigma'$. 
\item Let $a'$ be an edge of $\cY_2$. If $a'= f(a)$ for an edge $a$ of $\cY_1$, we choose $h_{a'}= \alpha(h_a)$. Otherwise, we choose an arbitrary element $h_{a'}$ that satisfies the conditions described in Definition \ref{inducedcomplexofgroups}. 
\end{itemize}
This yields a complex of groups $G(\cY_2)$ over $Y_2$. We now define a morphism of complex of groups $F=(F_\sigma, F(a)): G(\cY_1) \ra G(\cY_2)$ over $f$ as follows: 
\begin{itemize}
\item The maps $F_\sigma: G_\sigma \ra G_{f(\sigma)}$ are the monomorphisms $\alpha: \mbox{Stab}(\overline{\sigma}) \ra \mbox{Stab}(\overline{f}(\overline{\sigma}))$,
\item The elements $F(a)$ are all trivial.
\end{itemize} 
It is straightforward to check that this indeed yields an immersion $F=(F_\sigma, F(a)): G(\cY_1) \ra G(\cY_2)$ over $f$.

\begin{definition}
We call the immersion $F=(F_\sigma, F(a)): G(\cY_1) \ra G(\cY_2)$ over $f$ an immersion \textit{associated to} $(Y_1 \overset{f}{\ra} Y_2, X_1 \overset{\bar{f}}{\ra} X_2, G_1 \overset{\alpha}{\ra} G_2)$.
\end{definition}

\subsection{Amalgamation of non-positively curved complexes of groups.}

In what follows, $Y$ is a finite simplicial complex, $Y_1, Y_2$ are subcomplexes of $Y$ whose interiors cover $Y$, and $Y_0= Y_1 \cap Y_2$. We assume that $Y_0$ is connected. We further assume that, for $i=0,1, 2$, we are given a simplicial complex $X_i$, a $G_i$ a group acting without inversion on $X_i$, $p_i: X_i \ra Y_i$ a simplicial map factoring through $X_i/G_i$ and inducing a simplicial isomorphism $X_i/G_i \simeq Y_i$. 
We assume that, for $i=1,2$, we are given a homomorphism $\alpha_i: G_0 \ra G_i$ and an equivariant simplicial immersion $\bar{f_i}: X_0 \ra X_i$ over the inclusion $f_i: Y_0 \hra Y_i$ such that for every simplex $\sigma$ of $X_0$, the induced map $\alpha_i: \mbox{Stab}(\sigma) \ra \mbox{Stab}(\bar{f_i}(\sigma))$ is an isomorphism.\\

\begin{figure}[H]
\center
\scalebox{0.8}{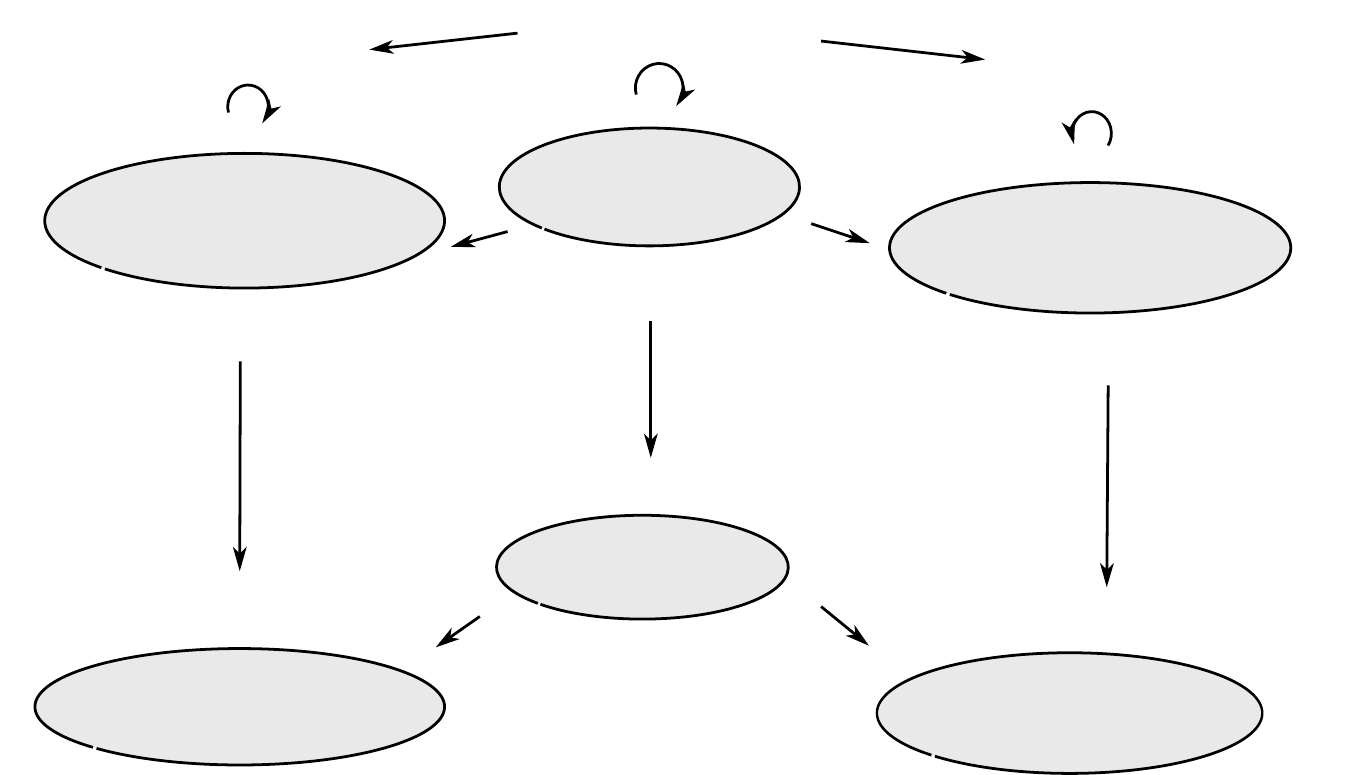}
\caption{A diagram of maps.}
\label{diagrammaps}
\end{figure}

By the results of the previous paragraph, we can choose complexes of groups $G(\cY_i)$ over $Y_i$ associated to these actions in such a way that there are immersions $G(\cY_0) \overset{F_i}{\ra} G(\cY_i)$ associated to $(Y_0 \overset{f_i}{\ra} Y_i, X_0 \overset{\bar{f_i}}{\ra} X_i, G_0 \overset{\alpha_i}{\ra} G_i)$. Note that the local maps $(F_i)_\sigma$ are isomorphisms.

We now use these immersions to amalgamate $G(\cY_1)$ and $G(\cY_2)$ along $G(\cY_0)$. So as to emphasise which complex of groups is under consideration, we will indicate it as a superscript (see below). We define a complex of groups $G(\cY)$ over $Y$ as follows: 

\begin{itemize}
\item If $\sigma$ is a vertex of $\cY_0$, we set $G^{Y}_\sigma = G^{Y_0}_\sigma$. 

If $\sigma$ is a vertex of $\cY_i \setminus \cY_0$, we set $G^{Y}_\sigma = G^{Y_i}_\sigma$.
\item If $a$ is an edge of $\cY_0$, we set $\psi^{Y}_a = \psi^{Y_0}_a$. 

If $a$ is an edge of $\cY_i \setminus \cY_0$, we set $\psi^{Y}_a = \psi^{Y_i}_a$. 

If $a$ is an edge of $\cY_i$ such that $i(a)$ is a vertex of $\cY_i \setminus \cY_0$ and $t(a)$ is a vertex of $\cY_0$, we set $\psi^{Y}_a = \big( (F_i)_{t(a)} \big)^{-1} \circ \psi^{Y_i}_a$. 
\item  If $(a,b)$ is a pair of composable edges of $\cY_0$, we set $g^{Y}_{a,b} = g^{Y_0}_{a,b}$. 

If $(a,b)$ is a pair of composable edges of $\cY_i \setminus \cY_0$, we set $g^{Y}_{a,b} = g^{Y_i}_{a,b}$. 

If $(a,b)$ is a pair of composable edges of $\cY$ such that $b$ is not an edge of $\cY_0$ but $t(a)$ is a vertex of $\cY_0$, we set $g^{Y}_{a,b} = \big( (F_i)_{t(a)} \big)^{-1}(g^{Y_i}_{a,b})$. 

\end{itemize}

\noindent It is straightforward to check all the compatibility conditions of a complex of groups.

\begin{definition}[Amalgamation of complexes of groups] We denote by $G(\cY_1) \underset{G(\cY_0)}{\cup} G(\cY_2)$ the previous complex of groups.
\end{definition}

\begin{thm}[Seifert-van Kampen Theorem for complexes of groups, Theorem  III.$\cC$.3.11.(5) of \cite{BridsonHaefliger}] With the same notations as above, the fundamental group of $G(\cY_1) \underset{G(\cY_0)}{\cup} G(\cY_2) $ is isomorphic to the pushout $G_1 \underset{G_0}{*} G_2.$ 

Moreover, let $F: G(\cY_1) \ra G(\cY) $ the inclusion of complex of groups over the inclusion $Y_1 \ra Y$. For every vertex $v$ of $Y_1$, the associated morphism $F_*: \pi_1(G(\cY_1), v) \ra \pi_1(G(\cY), v)$ is conjugated to the projection $G_1 \ra G_1 \underset{G_0}{*} G_2$.\qed
\label{VanKampen}
\end{thm}

We now assume in addition that $Y$ comes equipped with a piecewise-Euclidean structure. This endows $X_0, X_1, X_2$ with a piecewise-Euclidean simplicial structure that turns the maps $\overline{f_i}: X_0 \ra X_i$ into local isometries. Let $v$ be a vertex of $Y$. Since the interiors of $Y_1$ and $Y_2$ cover $Y$, the star of $v$ is fully contained in one of these subcomplexes. We thus obtain from Theorem \ref{nonpositivelycurveddevelopable} the following developability theorem: 

\begin{thm} Under the same assumptions as above, if $X_1$ and $X_2$ are CAT(0) for their induced piecewise-Euclidean structure, then $G(\cY_1) \underset{G(\cY_0)}{\cup} G(\cY_2)$ is non-positively curved, hence developable. \qed
\end{thm}

As stated at the beginning of this section, this theorem generalises directly to the case of a finite complex $Y$ covered by the interiors of a finite family of subcomplexes $(Y_i)$ such that the nerve of the associate open cover is a tree. We give the following particular case, which will be used in the article. 

Let $Y$ be a finite simplicial complex endowed with a piecewise-Euclidean structure, let $Y_0, Y_1, \ldots, Y_n$ be a family of connected subcomplexes of $Y$ whose interiors cover $Y$, and for $i=1, \ldots, n$, let $Y_i' = Y_0 \cap Y_i$. We assume that each $Y_i'$ is non empty and connected. We further assume that for $1 \leq i \neq j \leq n$, we have $Y_i \cap Y_j = \varnothing$. 

For each $Y_i$ (resp. $Y_i'$), we are given a simplicial complex $X_i$ (resp. $X_i'$), a group $G_i$ (resp. $G_i'$) acting without inversion on $X_i$ (resp. $X_i'$), $p_i: X_i \ra Y_i$ (resp. $p_i': X_i' \ra Y_i'$) a simplicial map factoring through $X_i/G_i$) (resp. $X_i'/G_i'$) and inducing a simplicial isomorphism $X_i/G_i \simeq Y_i$ (resp. $X_i'/G_i' \simeq Y_i'$). This yields a piecewise-Euclidean structure on each $X_i$. 
We assume that, for $i=1,\ldots, n$, we are given homomorphisms $\alpha_i: G_i' \ra G_0$ and $\beta_i: G_i' \ra G_i$, an $\alpha_i$-equivariant simplicial immersion $\overline{f_i}: X_i' \ra X_0$ over the inclusion $Y_i' \hra Y_0$ and a $\beta_i$-equivariant simplicial immersion $\overline{g_i}: X_i' \ra X_i$ over the inclusion $Y_i' \hra Y_i$. We finally assume that for every simplex $\sigma$ of $X_i'$, the induced maps $\alpha_i: \mbox{Stab}(\sigma) \ra \mbox{Stab}(\bar{f_i}(\sigma))$ and $\beta_i: \mbox{Stab}(\sigma) \ra \mbox{Stab}(\bar{g_i}(\sigma))$ are isomorphisms.

As before, we can construct induced complexes of groups $G(\cY_i)$ over $Y_i$ and $G(\cY_i')$ over $Y_i'$, along with immersions $G(\cY_i') \ra G(\cY_0)$ and $G(\cY_i') \ra G(\cY_i)$. These complexes of groups can in turn be amalgamated to obtain a complex of groups $G(\cY)$ over $Y$. We get the following: 

\begin{thm}
If each simplicial complex $X_i$, $i=0, \ldots,n$, is CAT(0) for its induced piecewise-Euclidean structure, then $G(\cY)$ is non-positively curved, hence developable.\qed
\label{CartanHadamard}
\end{thm}

\section{Rotation families, actions on trees and geometric small cancellation theory.}
\label{smallcancellation}

We now present the theory of small cancellation over a graph of groups, using the notion of rotation families introduced by Gromov \cite{GromovRandomWalkRandomGroups}.

\begin{definition}[rotation family]
Let $X$ be a metric space and $G$ be a group acting isometrically on $X$. A \textit{rotation family} consists of a pairwise distinct collection $(A_\xi, H_\xi)_{\xi \in \Xi}$ where:
\begin{itemize}
\item each $H_\xi$ is a subgroup of $G$ stabilising the subspace $A_\xi \subset X$ (i.e. $H_\xi A_\xi = A_\xi$),
\item there is an action of $G$ on $\Xi$ which is compatible with the action of $G$ on $X$, i.e. for every $\xi \in \Xi$, we have $A_{g\xi} = g A_\xi$ and $H_{g\xi}= gH_\xi g^{-1}$.
\end{itemize}
\label{rotationfamily}
\end{definition}

In what follows, we will deal only with actions on trees. For the related notions in the general case, we refer to \cite{CoulonAsphericity, DelzantGromovCourbureMesoscopique, DahmaniGuirardelOsin}.

\begin{definition}[Geometric small cancellation condition for actions on trees]
Let $T$ be a simplicial tree, $G$ a group acting isometrically on $T$, and $(A_\xi, H_\xi)_{\xi \in \Xi}$ a rotation family such that each $H_\xi$ is an infinite cyclic subgroup generated by a hyperbolic element $g_\xi$ with axis $A_\xi\subset T$. We define the following constants: 
$$ l_{\mathrm{max}} = \underset{\xi \neq \xi'}{\mbox{max~}} \mbox{diam} (A_\xi \cap A_{\xi'} ),$$
$$ R_{\mathrm{min}} = \underset{\xi \in \Xi}{\mbox{min }} l( g_\xi),$$
where $l(g_\xi)$ denotes the translation length of $g_\xi$ acting on $T$.

We say that the rotation family $(A_\xi, H_\xi)_{\xi \in \Xi}$ satisfy the \textit{geometric small cancellation condition} $C''(\alpha)$ if 
$$ l_{\mathrm{max}} < \alpha R_{\mathrm{min}}.$$
\end{definition}

In what follows, we will be particularly interested in the $C''(\alpha)$ condition for $\alpha \leq \frac{1}{6}$.

\begin{rmk}[Classical small cancellation] This geometric framework covers the classical $C''(1/6)$ small cancellation theory. Let $F_n$ be the free group on $n$ generators, acting freely cocompactly on the associated $2n$-valent tree $T_n$. To every element $g$ of $F_n$ corresponds an isometry of $T_n$. Let $\cR= (g_i)_{i \in I} $ be a set of cyclically reduced words of $F_n$. We assume that $\cR$ is \textit{symmetrised}, that is, inverses and cyclic conjugates of elements of $\cR$ belong to $\cR$. Each $g_i$ defines an axis $A_i \subset T_n$ and a non-trivial infinite subgroup $\langle g_i \rangle \subset F_n$. The family of translates $gA_i, g \in F_n$ and subgroups $g \langle g_i \rangle g^{-1}$ defines a rotation family, and one sees that the geometric small cancellation condition $C''(1/6)$ is equivalent to requiring that the length of any common prefix to two elements of $\cR$  is strictly less than a sixth of the length of the shortest element of $\cR$.
\end{rmk}

\begin{rmk}[$C'(1/6)$ versus $C''(1/6)$] Another small cancellation condition which is more commonly used in geometric group theory is the $C'(1/6)$-condition which, restated in this geometric framework of actions on trees, requires that for $\xi \neq \xi'$, the diameter of the intersection $A_\xi \cap A_{\xi'}$ is strictly less than $\frac{1}{6} \cdot \mbox{min}(l(g_\xi), l(g_{\xi'}))$. In the case of classical small cancellation, the uniform condition $C''(1/6)$ is much stronger than condition $C'(1/6)$. Indeed, a symmetrised set satisfying the $C''(1/6)$ condition is necessarily finite. In contrast, infinitely-presented $C'(1/6)$ small cancellation groups form a very rich class of groups \cite{ArzhantsevaDrutuInfiniteSmallCancellation, ArzhantsevaOsajdaSmallCancellationHaagerup}. In this article, the uniform bound in the $C''(1/6)$ condition will be crucial for certain geometric constructions (see Sections \ref{Gromovconstruction} and \ref{CAT0}).
\end{rmk}

\section{A non-positively curved complex of groups arising from geometric small cancellation theory.}
\label{Gromovconstruction}

From now on, we consider a graph of groups $G(\Gamma)$ over a finite graph $\Gamma$, with fundamental group $G$ and associated Bass--Serre tree $T$, together with a rotation family $(A_\xi, H_\xi)_{\xi \in \Xi}$ satisfying the assumptions of Theorem \ref{mainPS1}. Recall in particular that this implies that $ l_{\mathrm{max}} < \frac{1}{6}R_{\mathrm{min}}$.

\begin{definition}[cones]  For every $\xi \in \Xi$, let $C_\xi$ be the cone with apex $O_\xi$ over the axis $A_\xi$, with a simplicial structure coming from that of $A_\xi$ and in which every triangle is modelled after a flat isosceles triangle $\tau$ whose basis has length $1$ and whose other edges have length $r = \big( 2 \sin(\frac{\pi}{R_{\mathrm{min}}}) \big)^{-1}$, the angle at the apex $O_\xi$ being $\frac{2\pi}{R_{\mathrm{min}}} $.
\end{definition}

\begin{definition}[coned-off space] Let $\widehat{T}$ be the 2-dimensional simplicial complex obtained from the disjoint union of $T$ and the various cones $C_\xi$ by identifying the base of each $C_\xi$ with the axis $A_\xi \subset T$ in the obvious way. This complex is endowed with a structure of piecewise-Euclidean complex. The action of $G$ on $T$ naturally extends to an isometric action on $\widehat{T}$. 
\label{coneoff}
\end{definition}
\begin{rmk} In \cite{GromovCATkappa}, Gromov uses a different metric on the coned-off space. In this article, we use a piecewise-Euclidean metric in order to use the combination theorems of \cite{MartinBoundaries} for groups acting on CAT(0) piecewise-Euclidean complexes (or more generally on $M_\kappa$-complexes in the sense of Bridson \cite{BridsonPhD}).
\end{rmk}

\noindent \textbf{Coordinates.} We introduce some coordinates on $\widehat{T}$ as follows. For an element $u \in C_\xi \subset \widehat{T},$ we write  $u = [\xi, x, t]$, where $x$ is the intersection point of the ray $[O_\xi, u)$ with $T$, and $t = d(u, O_\xi)$.\\

 Since axes $A_\xi$ can share more than one edge, the coned-off space $\widehat{T}$ is not CAT(0) in general.  In  what follows, we generalise a trick due to Gromov \cite{GromovCATkappa} to turn this is space into a CAT(0) space. This will be done by performing certain identifications described below, which rely on the following definition. 

\begin{definition}[Critical angle $\theta_c$] We define the \textit{critical angle} 
$$\theta_c = \frac{\pi}{2} - \pi \frac{ l_{\mathrm{max}}}{R_{\mathrm{min}}}.$$
\end{definition}
Note that the $C''(1/6)$ condition immediately implies the following: 

\begin{lem} We have $\theta_c >\frac{\pi}{3}$.  \qed
\label{lemmetechnique1}
\end{lem}

\noindent \textbf{Identification of slices.} Let $\xi \neq\xi'$ be two elements of $\Xi$ such that the intersection $ I_{\xi, \xi'}\subset T:=A_\xi \cap A_{\xi'}$ is non-empty. By the $C''(1/6)$ condition, $I_{\xi, \xi'}$ is a segment $[x_1,x_2]$ of $T \subset \widehat{T}$ of length at most $l_{\mathrm{max}}$. In particular, the subset $C(I_{\xi, \xi'})$ of $C_\xi$ consisting of the cone over $I_{\xi, \xi'}$ isometrically embeds in the plane. Given point $x,y,z$ of $C(I_{\xi, \xi'})$,  we can thus speak of the angle $\angle_{z} (x,y)$ between $x$ and $y$ seen from $z$,.

\begin{definition}[slice]
Let $\xi$ and $\xi'$ be two distinct elements of $\Xi$ such that the associated axes have a nonempty intersection $ I_{\xi, \xi'} \subset T$. We define the \textit{slice} $C_{\xi, \xi'} \subset C_\xi$ as the set of points $u=[\xi, x, t]$ of $C_\xi$ such that $x \in I_{\xi, \xi'}$ and both angles $\angle_{x_1} (O_\xi, u) $ and $\angle_{x_2} (O_\xi, u) $ are at least $\theta_c$.
\end{definition}

\begin{rmk}
Note in the previous definition that $\theta_c$ corresponds to the angle $\angle_{x_1} (O_\xi, x_2) $ when the segment $[x_1,x_2]$ has exactly $l_{\mathrm{max}}$ edges.
\end{rmk}

We have the following result:

\begin{lem} A point $[\xi,x,t]$ of a cone contained in a slice satisfies the inequality $t \geq r \sin(\theta_c) > \frac{\sqrt{3}}{2}r$. \qed
\label{lemmetechnique2}
\end{lem}

\begin{proof}
Recall that $\theta_c > \frac{\pi}{3}$ by Lemma \ref{lemmetechnique1}. The result is then an elementary application of triangle geometry.
\end{proof}

 Following an idea of Gromov \cite{GromovCATkappa}, we now glue slices $C_{\xi, \xi'} \subset C_\xi$ and $C_{\xi', \xi} \subset C_{\xi'}$ together by identifying points $u=[\xi, x, t] \in C_{\xi, \xi'}$ and $u'=[\xi', x', t'] \in C_{\xi', \xi}$ if $x=x' \in I$, $t=t'$.\\

\begin{figure}[H]
\center
\scalebox{0.75}{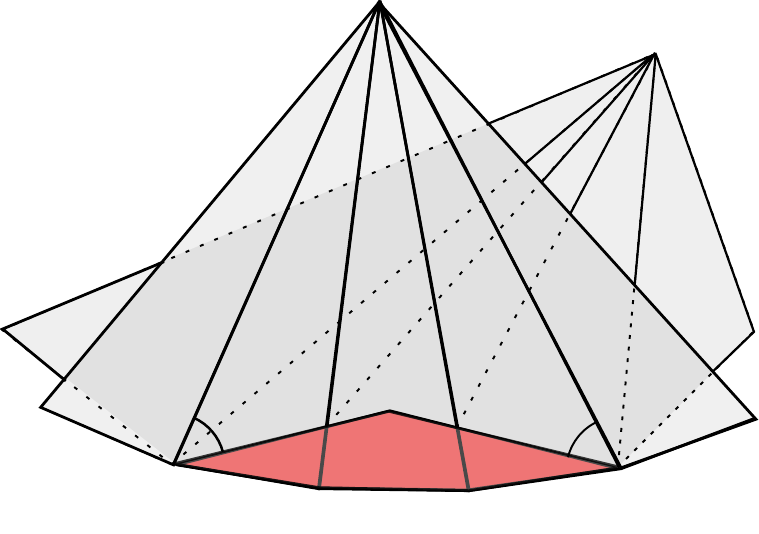}
\caption{A slice identification.}
\label{sliceidentification}
\end{figure}

\begin{definition}[modified coned-off space] Let $\widehat{T}_{\mathrm{slices}}$ be the space obtained from $\widehat{T}$ by making such identifications for every distinct $\xi, \xi' \in \Xi$ such that the associated axes $A_\xi, A_{\xi'}$ have a nontrivial intersection. We call it the \textit{modified coned-off space}.

Let also $Y_\mathrm{slices}$ be the quotient complex $\lquotient{G}{\widehat{T}_{\mathrm{slices}}}$. This space can be seen as the graph $\lquotient{G}{T}$ with a collection of polygons attached and partially glued together along slices, a polygon corresponding to the image of a cone of $\widehat{T}$. 
\end{definition}

We now turn to the construction of a non-positively curved complex of groups with $~\rquotient{G}{\ll H_\xi \gg }$ as fundamental group. This complex of groups will be defined using the amalgamation procedure described in Section \ref{trees}.  We start by defining various subcomplexes.

\begin{definition}[polygonal neighbourhoods, ribbons] Let $N_\xi$ (resp. $N_\xi'$) be the \textit{polygonal neighbourhood} of the apex $O_\xi \in C_\xi$ in $\widehat{T}$ obtained by taking the image of $C_\xi$ under the homothety of centre $O_\xi$ and ratio $\frac{1}{2}$ (resp. $\frac{1}{4}$). 

We also define the \textit{ribbon} $R_\xi$ as the subcomplex obtained from $N_\xi$ by deleting the interior of $N_\xi'$. 
\end{definition} 
 
Up to making simplicial subdivisions, we can assume that the various complexes $N_\xi$, $N_\xi'$ and $R_\xi$ are subcomplexes of $\widehat{T}$. By Lemma \ref{lemmetechnique2} we can identify them with their images in $\widehat{T}_\mathrm{slices}$.

\begin{definition}[truncated cone-off spaces] We define the subcomplex $\overline{T} \subset \widehat{T}$ (resp. $\overline{T}_{\mathrm{slices}} \subset \widehat{T}_\mathrm{slices})$ as the subcomplex obtained from $\widehat{T}$ (resp.$\widehat{T}_{\mathrm{slices}}$) by deleting the interiors of all the polygonal neighbourhoods $N_\xi'$.  This subcomplex comes equipped with an action of $G$ by simplicial isometries. We also denote by $\overline{Y}$ (resp. $\overline{Y}_\mathrm{slices}$) the associated quotient space, which we see as a subcomplex of $Y$ (resp. $Y_\mathrm{slices}$).
\label{truncated}
\end{definition}

\noindent \textbf{A first group action.} We consider the action of $G$ on $\overline{T}_\mathrm{slices}$. \\

Let $\xi \in \Xi$ and $g_\xi$ be a generator of $H_\xi$. Since $g_\xi$ acts hyperbolically on $T$, we can write $g_\xi= h_\xi^{n_\xi}$ where $n_\xi \geq 1$ and $h_\xi \in G$ is not a proper power of an element of $G$. Note that $h_\xi$ also acts hyperbolically on $T$ with axis $A_\xi$. Let $S_\xi$ be the global stabiliser of $A_\xi$, that is, the set of elements $g\in G$ such that $gA_\xi = A_\xi$. Note that $S_\xi$ acts on $C_\xi$ by simplicial isometries.  It follows from condition (RF$_2$) that  $S_\xi$ is the infinite cyclic group generated by $h_\xi$.  Note that there is an isometric embedding $R_\xi \hra \overline{T}_{\mathrm{slices}}$ which is equivariant with respect to the inclusion $\alpha_\xi:S_\xi \ra G$.\\

\noindent \textbf{A second group action.} We consider the action of $S_\xi$ on the ribbon $R_\xi$.\\

Let $K_{\xi}$ be the quotient of the polygonal neighbourhood $N_\xi$ under the action of the subgroup of $S_\xi$ generated by $g_\xi$. This is a regular polygon with $l(g_\xi)=n_\xi\cdot l(h_\xi)$ edges. There is an action by isometries of the cyclic group  $S_\xi/\langle g_\xi \rangle$ of order $n_\xi$ on $K_\xi$ by rotation of $l(h_\xi)$ triangles.  \\

\noindent \textbf{A third group action.} We consider the action of $S_\xi/\langle g_\xi\rangle$ on $K_\xi$.\\

 Let $\beta_\xi: S_\xi \ra S_\xi / \langle g_\xi \rangle$ be the canonical projection. Then there is a $\beta_\xi$-equivariant local isometry $R_\xi \ra K_\xi$.

\begin{figure}[H]
\center
\scalebox{0.75}{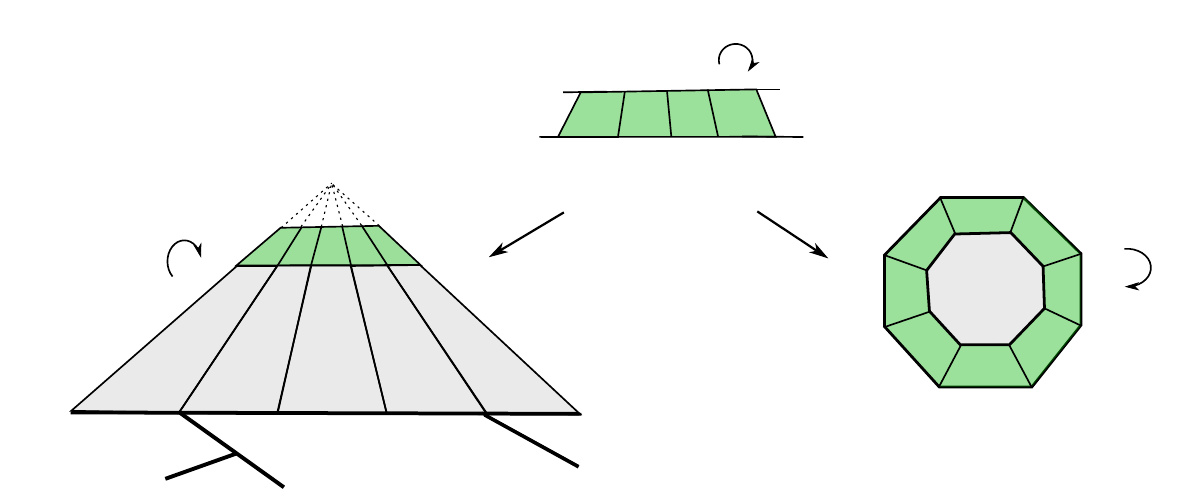}
\caption{The three group actions to be amalgamated.}
\end{figure}

Note that the $\alpha_\xi$-equivariant embedding $R_\xi \hra \overline{T}_{\mathrm{slices}}$ yields an isometric embedding $ ~\lquotient{S_\xi}{R_\xi} \hra ~\lquotient{G}{\overline{Y}_{\mathrm{slices}}}$. Moreover, the $\beta_\xi$-equivariant local isometry $R_\xi \ra K_\xi$ yields an isometric embedding $~\lquotient{S_\xi}{R_\xi} \hra  ~\lquotient{\big( S_\xi / \langle g_\xi \rangle \big)}{K_\xi}  $. The complex obtained from the disjoint union of $\overline{Y}_{\mathrm{slices}}$ and the various complexes $~\lquotient{\big( S_\xi / \langle g_\xi \rangle \big)}{K_\xi}  $, for a set of representatives of $\lquotient{G}{\Xi}$, by identifying the embedded copies $ ~\lquotient{S_\xi}{R_\xi} \hra ~\lquotient{G}{\overline{Y}_{\mathrm{slices}}}$ and   $~\lquotient{S_\xi}{R_\xi} \hra  ~\lquotient{\big( S_\xi / \langle g_\xi \rangle \big)}{K_\xi}  $ is naturally isometric to the quotient complex $Y_\mathrm{slices}$; we will thus think of these quotients as subcomplexes of $Y_\mathrm{slices}$. \\

\begin{definition}
Using the results of Section \ref{trees}, we can thus amalgamate these three actions to get a complex of groups $G(\cY_\mathrm{slices})$ over $Y_\mathrm{slices}$.
\label{THEcomplexofgroups}
\end{definition}

The following result will be proved in Section \ref{CAT0} by studying links of points of $\widehat{T}_\mathrm{slices}$.

\begin{prop}
 The simplicial complex $\widehat{T}_\mathrm{slices}$ is CAT(0).
\label{CAT(0)1}
\end{prop}

\begin{rmk} In the case of the free group on $k$ generators acting on its associated $2k$-valent tree in the natural way, Proposition \ref{CAT(0)1} implies that classical small cancellation groups satisfying the $C''(1/6)$-condition are CAT(0), as shown by Gromov \cite{GromovCATkappa}. From a different point of view, Wise \cite{WiseSmallCancellation} proved that many classical small cancellation groups act properly and cocompactly on CAT(0) cube complexes, using techniques that go back to work of Sageev \cite{SageevCubeComplex}.
\end{rmk}

We have the following:

\begin{lem}
 The simplicial complex $K_\xi$ is CAT(0).
\label{CAT(0)2}
\end{lem}

\begin{proof}
 The link of the centre of $K_\xi$ is a loop of length $l(g_\xi) \frac{2\pi}{R_{\mathrm{min}}} \geq 2 \pi$, so the result follows from a criterion due to Gromov \cite[Theorem II.5.5]{BridsonHaefliger}.
\end{proof}

\begin{thm}
 The complex of groups $G(\cY_\mathrm{slices})$ is non-positively curved, hence developable, and its fundamental group is isomorphic to $\rquotient{G}{\ll H_\xi \gg }$.
\label{NPCcomplexofgroups}
\end{thm}

\begin{proof}
Since $\widehat{T}_\mathrm{slices}$ and $K_\xi$ are CAT(0) by Proposition \ref{CAT(0)1} and Lemma \ref{CAT(0)2}, the complex of groups $G(\cY_\mathrm{slices})$ is non-positively curved, hence developable by Theorem \ref{CartanHadamard}. To compute the fundamental group of $G(\cY_\mathrm{slices})$ we can assume that $\lquotient{G}{\Xi}$ is reduced to a single element, the general case following in the same way. Choose an element $\xi \in \Xi$ and a generator $g_\xi=h^{n_\xi}$ of $S_\xi$ (with the same notations as before). It follows from the Van Kampen theorem \ref{VanKampen} that the fundamental group of $G(\cY_\mathrm{slices})$ is isomorphic to the amalgamated product $G *_{S_\xi} S_\xi / \langle g_\xi \rangle $, where the morphism $\alpha_g: S_\xi \hra G$ is the inclusion and the morphism $\beta_g: S_\xi \ra S_\xi / \langle g_\xi \rangle$ is the canonical projection. Thus this group is isomorphic to $\rquotient{G}{\ll g_\xi\gg}$, and the result follows. 
\end{proof} 

This theorem implies the following corollary, which is well-known for classical small cancellation over free products with amalgamation or HNN extensions (see \cite[Theorems V.11.2 and V.11.6]{LyndonSchupp}):

\begin{cor}
The quotient map $G \ra \rquotient{G}{\ll H_\xi \gg }$ embeds each local group of $G$. 
\end{cor}

\begin{proof}
Let $G(\overline{\cY}_\mathrm{slices})$ be the complex of groups over $\overline{Y}_\mathrm{slices}$ associated to the action of $G$ on $\overline{T}_\mathrm{slices}$. By construction, this complex of groups is the restriction of $G(\cY_\mathrm{slices})$ to the subcomplex $\overline{Y}_\mathrm{slices}$, that is, there exists a morphism of complexes of groups $F=(F_\sigma, F(a)): G(\overline{\cY}_\mathrm{slices}) \ra G(\cY_\mathrm{slices})$ over the inclusion $\overline{Y}_\mathrm{slices} \hra Y_\mathrm{slices}$ such that each local map $F_\sigma: G_\sigma \ra G_\sigma$ is the identity and all the elements $F(a)$ are trivial. For a chosen basepoint $v_0 \in \overline{Y}_\mathrm{slices}$, the morphism $F$ induces a map $F_*:\pi_1( G(\overline{\cY}_\mathrm{slices}), v_0) \ra \pi_1( G(\cY_\mathrm{slices}), v_0)$  which is conjugated to $G \ra \rquotient{G}{\ll H_\xi \gg }$ by Theorem \ref{VanKampen}. As $G(\cY_\mathrm{slices})$ is developable, the maps $(F_{G(\cY_\mathrm{slices}), v_0})_\sigma: G_\sigma \ra \rquotient{G}{\ll H_\xi \gg }$ are injective and factor as $G_\sigma \hra G \twoheadrightarrow  \rquotient{G}{\ll H_\xi \gg }$ by Proposition \ref{developabilityalgebraic}, hence the result.
\end{proof}

Recall that a torsion element of a group acting by simplicial isometries without inversion on a CAT(0) space necessarily fixes a vertex \cite[Corollary II.2.8]{BridsonHaefliger}. Hence the CAT(0) geometry of the universal cover of $G(\cY)$ yields a geometric proof of the following result, which is well-known for classical small cancellations over free products \cite[Theorem V.10.1]{LyndonSchupp}:

\begin{cor}
Let $g$ be a torsion element in $\rquotient{G}{\ll H_\xi \gg }$, then either 
\begin{itemize}
\item[$(i)$] $g$ is conjugate to the projection of a torsion element in a local factor of $G$, or
\item[$(ii)$] $g$ is conjugate to the projection of an element a power of which is in $ H_\xi $. \qed
\end{itemize} 
\end{cor} 

We now study the local groups of the complex of groups $G(\cY_\mathrm{slices})$. We first need the following definition. 

\begin{definition}
Let $u$ be a point of $\overline{T}_\mathrm{slices}$ and choose a element $[\xi,x,t]\in \overline{T}$ that projects to $u$. Suppose that there exists an element $\xi' \neq \xi$ such that $u$ belongs to the slice $C_{\xi, \xi'}$. Then there exist exactly two points $a , b$ in $A_\xi \subset T$ such that the angle between $O_\xi$ and $u$ seen from $a$ and $b$ is exactly the critical angle $\theta_c$. Let $I_u$ be the minimal subcomplex of the axis $A_\xi$ containing the geodesic between $a$ and $b$. 
\label{stabinterval}
\end{definition}

\begin{lem}
Let $u$ be a point of $\overline{T}_\mathrm{slices}$. The stabiliser of $u$ is non-trivial only if $u$ belongs to some slice $C_{\xi, \xi'} \subset \overline{T}_\mathrm{slices}$ for some $\xi \neq \xi'$ in $\Xi$, in which case that stabiliser is exactly the global stabiliser of $I_u$.
\label{stabiliserinterval}
\end{lem}

\begin{proof}
Let $u$ be a point of $\overline{T}_\mathrm{slices}$ and choose a element $[\xi,x,t]\in \overline{T}$ that projects to $u$. For an element $g$  of $G$,  $g$ sends $[\xi,x,t]$ to $[g\xi,gx,t]$. If $g$ fixes $u \in  \overline{T}_\mathrm{slices}$, then these two points have to be identified, which implies that $gx=x$. If $g\xi=\xi$, then $g$ stabilises the axis $A_\xi$. By condition (RF$_2$), the only element of $G$ fixing a point of $C_\xi$, hence of $A_\xi$, is the identity, which implies that $g$ is the identity of $G$. If $g\xi \neq \xi$, then $[\xi,x,t]$ is contained in the slice $C_{\xi, g\xi}$. In particular, the segment $I_u$ is contained in $A_\xi \cap gA_\xi$ and $g$ globally stabilises $I_u$. Reciprocally, an element stabilising $I_u$ fixes $u$ since $G$ acts on $\overline{T}_\mathrm{slices}$ by isometries.
\end{proof}

\begin{lem}
For every $\xi \in \Xi$, the action of $S_\xi / \langle g_\xi \rangle$ on the polygon $K_\xi$ is free on the complement of its centre. The stabiliser of its centre is exactly $S_\xi / \langle g_\xi \rangle$. \qed
\label{stabpolygon}
\end{lem}

\begin{cor}
The local groups of $G(\cY_\mathrm{slices})$ are either finite or subgroups of the local groups of $G(\Gamma)$.
\end{cor}

\begin{proof}
Since $G(\cY_\mathrm{slices})$ is an amalgam by Definition \ref{THEcomplexofgroups}, it is enough to look at stabilisers for the action of $G$ on $\overline{T}_\mathrm{slices}$ and for the action of $S_\xi / \langle g_\xi \rangle$ on $K_\xi$. For the latter action, the result follows from Lemma \ref{stabpolygon}. 

Since $G$ acts on $T$ without inversion, it acts without inversion on $\overline{T}_\mathrm{slices}$. Thus, it is enough to look at stabilisers of points of $\overline{T}_\mathrm{slices}$ under the action of $G$. Let $u$ be a point of $\widehat{T}_\mathrm{slices}$ and $g$ an element of $G$ fixing $u$. The stabiliser of $u$ is either trivial or exactly the global stabiliser of $I_u$ by Lemma \ref{stabiliserinterval}. Since $G$ acts on $T$ without inversion, either $g$ fixes pointwise $I_u$ or $g$ flips it and fixes the central vertex of $I_u$. In each case, the stabiliser of $u$ is contained in the stabiliser of some vertex of $T$.
\end{proof}

\section{Construction of cocompact models of classifying spaces for proper actions.}
\label{EG}

Gluing slices together was used to prove that the complex of groups $G(\cY_\mathrm{slices})$ is non-positively curved. We now modify the construction so as to get a complex of groups that can be used to study the finiteness properties of the small cancellation quotient. In particular, we obtain the following: 

\begin{thm}
Let $G(\Gamma)$ be a graph of groups over a finite simplicial graph $\Gamma$, with fundamental group $G$ and Bass--Serre tree $T$.  Let $(A_\xi, H_\xi)_{\xi \in \Xi}$ a rotation family that satisfies conditions {\normalfont (RF$_1$)} and {\normalfont (RF$_2$)} of Theorem \ref{mainPS1}, as well as the geometric small cancellation condition $C''(1/6)$. If all the local groups of $G(\Gamma)$ admit cocompact models of classifying spaces for proper actions, then so does $\rquotient{G}{\ll H_\xi \gg }$.
\label{PScombinationEG}
\end{thm}

\begin{definition} Let $X_\mathrm{slices}$ be the universal covering of $G(\cY_\mathrm{slices})$ and $\widetilde{\Gamma}$ the preimage of $\Gamma$ under the projection $X_\mathrm{slices} \ra Y_\mathrm{slices}$.
\end{definition}

The complex $X_\mathrm{slices}$ can be thought as obtained in the following way. Recall that $Y_\mathrm{slices}$ is obtained from $\Gamma$ by attaching to it a bunch of polygons and identifying slices of such polygons. For such a polygon, any connected component of the preimage of its interior is the interior of a polygon of $X_\mathrm{slices}$. Such polygons of $X$ are glued together according to the same slice identifications procedure. 

Let $P$ be a polygon of $X_\mathrm{slices}$ and consider the polygonal neighbourhood of its apex which is the intersection of $P$ with the preimage of $Y_\mathrm{slices} \setminus \overline{Y}_\mathrm{slices}$ (green region in Figure \ref{radialcollapsing}). We now collapse radially the complement of this neighbourhood (grey region in Figure \ref{radialcollapsing}), simultaneously for every polygon $P$ of $X_\mathrm{slices}$. 

\begin{figure}[H]
\center
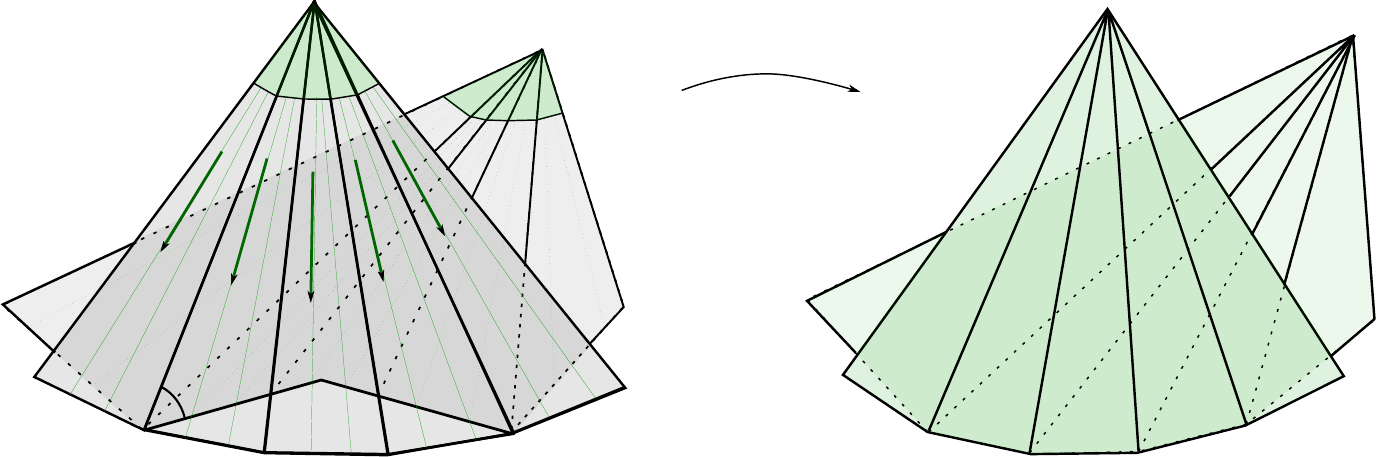
\caption{Radial collapsing.}
\label{radialcollapsing}
\end{figure}
 
\begin{definition} Let $X$ be the space obtained after such collapses and $p_\mathrm{coll}: X_\mathrm{slices} \ra X$ the associated \textit{collapsing map}. The space $X$ inherits from $X_\mathrm{slices}$ a canonical simplicial structure (see Figure \ref{radialcollapsing}). Identifying slices in $X$ yields a \textit{slice identification map} $p_\mathrm{slices}: X \ra X_\mathrm{slices}$. The action of $\rquotient{G}{\ll H_\xi \gg }$ on $X_\mathrm{slices}$ yields an action on $X$ which makes $p_\mathrm{slices}$ equivariant, and we denote by $Y$ the quotient space.
\end{definition}

The space $X$ space is topologically the graph $\widetilde{\Gamma}$ with a bunch of polygons glued to it. Note that $Y$ is obtained from $Y_\mathrm{slices}$ by applying the same collapsing procedure. It is the graph $\Gamma$ with a collection of polygons attached to it. As this can be done without loss of generality, we will consider for the remaining of this section that this collection is reduced to a single polygon, so as to lighten notations.

\begin{lem}
Let $H$ be a non-trivial  finite subgroup of $\rquotient{G}{\ll H_\xi \gg }$. Then the fixed-point set $X^H$ is non-empty and contractible.
\label{contractiblefixedpoint}
\end{lem}

\begin{proof}
Let $H$ be a non-trivial  finite subgroup of $\rquotient{G}{\ll H_\xi \gg }$. Since $X_\mathrm{slices}$ is CAT(0), $H$ fixes a point of $ X_\mathrm{slices}$ by \cite[Corollary II.2.8]{BridsonHaefliger}. The collapsing map $p_\mathrm{coll}: X_\mathrm{slices} \ra X$ being equivariant, $H$ also fixes a point of $X$, hence the fixed-point set $X^H$ is non-empty. 

Let $\gamma$ be a loop contained in the $1$-skeleton of the fixed-point set $X^H$.  Note that the restriction of the slice-identification map $p_\mathrm{slices}:X \ra X_\mathrm{slices} $ to the $1$-skeleton of $X$ is a homeomorphism onto its image. If $\gamma$ was an embedded loop, then its image in $X_\mathrm{slices}$ would be an embedded loop fixed by $H$. As $X_\mathrm{slices}$ is CAT(0), the fixed-point set $X_\mathrm{slices}^H$ would be convex by \cite[Corollary II.2.8]{BridsonHaefliger}, thus there would exist a $2$-cell of $X_\mathrm{slices}$ containing the centre of a polygon of $X_\mathrm{slices}$ and fixed by $H$, which contradicts Lemma \ref{stabpolygon}. Thus, the $1$-skeleton of $X^H$ contains no embedded loop, hence $X^H$ is contractible.
\end{proof}

\begin{proof}[Proof of Theorem \ref{PScombinationEG}] This follows from the previous lemma together with Theorem \ref{combinationEG}.
\end{proof}

\section{Hyperbolicity of the small cancellation quotient}
\label{hyperbolicityquotient}
In this section we prove the following: 

\begin{thm}
There exists a universal constant $0<\lambda_\mathrm{univ}\leq \frac{1}{6}$ such that the following holds. Let $G(\Gamma)$ be a graph of groups over a finite simplicial graph, with fundamental group $G$ and Bass--Serre tree $T$. Let $(A_\xi, H_\xi)_{\xi \in \Xi}$ a rotation family satisfying conditions {\normalfont(RF$_1$)} and {\normalfont (RF$_2$)}. Suppose in addition that:
\begin{itemize}
\item[{\normalfont (HC$_1$)}] The local groups of $G(\Gamma)$ are hyperbolic and all the local maps are quasiconvex embeddings,
\item[{\normalfont (HC$_2$)}] The action of $G$ on $T$ is acylindrical, 
\item[{\normalfont (HC$_3$)}] There are only finitely many elements in $\Xi$ modulo the action of $G$. 
\end{itemize}
If $(A_\xi, H_\xi)_{\xi \in \Xi}$ satisfies the geometric small cancellation condition $C''(\lambda_\mathrm{univ})$, then the quotient group $\rquotient{G}{\ll H_\xi \gg }$ is hyperbolic and the projection $G \ra \rquotient{G}{\ll H_\xi \gg }$ embeds each local group of $G(\Gamma)$ as a quasiconvex subgroup.
\label{PScombinationhyperbolic}
\end{thm}

\begin{definition}[Acylindricity]
The action of $G$ on $T$ is said to be \textit{acylindrical} if there exists a uniform upper bound on the distance between two points which are fixed by an infinite subgroup of $G$.
\end{definition}

We assume that $\widehat{T}_\mathrm{slices}$ is given a simplicial structure such that $G$ acts on it by simplicial isomorphisms.

\begin{lem}
The pointwise stabiliser of every simplex of $\widehat{T}_\mathrm{slices}$ is hyperbolic. Moreover, for every inclusion of simplices $\sigma \subset \sigma'$, the inclusion $G_{\sigma'} \hra G_{\sigma}$ is a quasiconvex embedding.
\label{quasiconvexembedding}
\end{lem}

Before proving the lemma, we recall the following elementary result: 

\begin{lem}[Corollary 9.2 of \cite{MartinBoundaries}]
Let $I' \subset I$ be intervals contained in $T$. Then $\cap_{v \in I} G_v$ is hyperbolic and quasiconvex in $\cap_{v \in I'} G_v$. \qed
\end{lem}

\begin{proof}[Proof of Lemma \ref{quasiconvexembedding}]
Since $G$ acts on $\widehat{T}_\mathrm{slices}$ without inversion, we consider point stabilisers. Let $u$ be a point of $\widehat{T}_\mathrm{slices}$ and consider the interval $I_u \subset T$ defined in \ref{stabinterval}. Let $v_1, \ldots, v_n$ be the vertices of $I_u$. The stabiliser of $u$ is the global stabiliser of $I_u$ by Lemma \ref{stabiliserinterval}, so it contains $\bigcap_{1 \leq i \leq n} G_{v_i}$ as an index $\leq 2$ subgroup. The result now follows from Lemma \ref{quasiconvexembedding}.
\end{proof}

\begin{definition}[truncated universal cover]
As in Definition \ref{truncated}, it is possible to write the complex of groups $G(\cY)$ as a tree of complexes of groups, with one piece corresponding to a complex of groups $G(\overline{\cY})$ over the subcomplex $\overline{Y}$ obtained from $Y$ by removing small polygonal neighbourhoods of the centres of polygons of $Y$, and such that  $G(\overline{\cY})$ is isomorphic to the complex of groups obtained from the action of $G$ on $\overline{T}$. We define the \textit{truncated universal cover} $\overline{X}$ as the preimage of $\overline{Y}$ under the quotient map $X \ra Y$.
\end{definition}
\begin{lem}
The truncated universal cover $\overline{X}$ is $\rquotient{G}{\ll H_\xi \gg }$-equivariantly isomorphic to the quotient of the truncated coned-off space $\rquotient{\overline{T}}{\ll H_\xi \gg }$.
\label{equivisomorphic}
\end{lem}

\begin{proof}
We have a morphism $G(\overline{\cY}) \ra G$ whose associated development is $\overline{T}$. The morphism $G(\cY) \ra \rquotient{G}{\ll H_\xi \gg }$, whose development is the universal cover $X$, restricts to a morphism $G(\overline{\cY}) \ra \rquotient{G}{\ll H_\xi \gg }$ whose associated development is $\overline{X}$. As the kernel of the projection map $G \ra \rquotient{G}{\ll H_\xi \gg }$ is exactly $\ll  H_\xi  \gg$, the result follows from \cite[Theorem III.$\cC$.2.18]{BridsonHaefliger}.
\end{proof}

We recall the following properties of actions of rotation families:

\begin{prop}
There exists a universal constant $0<\lambda_\mathrm{univ}\leq \frac{1}{6}$ such that the following holds. Suppose that the rotation family $(A_\xi, H_\xi)_{\xi \in \Xi}$  satisfies the geometric $C''(\lambda_\mathrm{univ})$ small cancellation condition. Then:
\begin{itemize}
\item the complex $\rquotient{T}{\ll H_\xi\gg }$, and hence $\rquotient{\widehat{T}}{\ll H_\xi\gg }$, is hyperbolic;
\item if the action of $G$ on $T$ is acylindrical, then the action of $\rquotient{G}{\ll H_\xi\gg }$ on $\rquotient{T}{\ll H_\xi\gg }$ is also acylindrical. 
\end{itemize}
\label{DahmaniGuirardelOsin}
\end{prop}

\begin{proof}
The first part is a direct consequence of \cite[Proposition 2.17]{DahmaniGuirardelOsin}.
The second part follows from rewriting the proof of \cite[Proposition 5.22]{DahmaniGuirardelOsin} using the notion of acylindricity used in this article. While this can be done without any essential change, we briefly explain how to do so. 

First notice that since $G$ acts acylindrically on $T$, it also acts acylindrically on $\widehat{T}$. Let $\overline{a}, \overline{b}$ be points of  $\rquotient{\widehat{T}}{\ll H_\xi\gg }$ and $\overline{g}$  be an element of $\rquotient{G}{\ll H_\xi\gg }$ fixing them both. Note that it is enough to consider the case of points which are images of vertices of $T$ in the quotient $\rquotient{\widehat{T}}{\ll H_\xi\gg }$. We choose lifts $a, b \in \widehat{T}$  such that $d(a,b) = d(\overline{a}, \overline{b})$, a lift $g$ of $\overline{g}$ such that $gb=b$ and an element $r$ of $\ll H_\xi \gg$ such that $ga=ra$. If $r$ is the trivial element, then the result follows from the acylindricity of the action of $G$ on $\widehat{T}$. Suppose by contradiction that $r$ is non-trivial. Dahmani--Guirardel--Osin show in that case that $g$ also stabilises some axis $A_\xi$. Condition (RF$_2$) implies that the only element of $g$ fixing pointwise such an axis is the identity element. Thus, $g$ induces a non-trivial translation on that axis and $g$ is hyperbolic, which contradicts the fact that $g$ fixes $b$. 
\end{proof}

\begin{proof}[Proof of Theorem \ref{PScombinationhyperbolic}]
The collapsing map $p_\mathrm{coll}:X_\mathrm{slices}\ra X$ and the slice-identification map $p_\mathrm{slices}: X \ra X_\mathrm{slices}$ are equivariant quasi-isometries and are quasi-inverses of one another. Thus, it follows from Lemma \ref{equivisomorphic} and Proposition \ref{DahmaniGuirardelOsin} that the action of $\rquotient{G}{\ll H_\xi\gg }$ on $X_\mathrm{slices}$ is acylindrical and the complex $X_\mathrm{slices}$ is hyperbolic. This space is also CAT(0) by Theorem \ref{NPCcomplexofgroups}. Finally, stabilisers of simplices of $X_\mathrm{slices}$ are hyperbolic and inclusions of such stabilisers are quasiconvex embeddings by Lemma \ref{quasiconvexembedding}. The result thus follows from Theorem \ref{combinationhyperbolic}.
\end{proof}

\section{The geometry of $\widehat{T}_\mathrm{slices}$.}
\label{CAT0}

We now prove Proposition \ref{CAT(0)1}. Note that $\overline{T}_\mathrm{slices}$ being homotopy equivalent to the Bass--Serre tree $T$, it is contractible, hence we only have to prove that it is locally CAT(0) \cite[Theorem II.4.1]{BridsonHaefliger}. Since $\overline{T}_\mathrm{slices}$ is a $2$-dimensional complex, it is enough to prove that injective loops in links of points of $\overline{T}_\mathrm{slices}$ have length at least $2 \pi$ \cite[Theorem II.5.5]{BridsonHaefliger}. As this condition is preserved by taking subcomplexes, it is thus enough to prove that the modified coned-off space $\widehat{T}_\mathrm{slices}$ itself is CAT(0). In the case of classical small cancellation theory, this result was proved by Vinet (unpublished).\\

There are three types of points in $\widehat{T}_\mathrm{slices}$: apices of cones of $\widehat{T}_\mathrm{slices}$, points in the Bass-Serre tree $T$ and points in the interior of a cone.\\

\noindent \textbf{Apex of a cone.} Each apex of a cone of $\widehat{T}_\mathrm{slices}$ has a link simplicially isomorphic to a bi-infinite line, hence the result.\\

\noindent \textbf{Point in the interior of a cone.} Let $u$ be a point that is in the interior of a cone but is not an apex. A neighbourhood of $u$ in $\widehat{T}_\mathrm{slices}$ is obtained from neighbourhoods of $u$ in the various cones containing it by gluing them together in an appropriate way. First notice that if does not belong to any slice $C_{\xi, \xi'}$, then a sufficiently small neighbourhood of $u$ is isometric to a small neighbourhood of the unique point of $\widehat{T}$ projecting to $u$. As cones of $\widehat{T}$ are CAT(0), the result follows. 

Suppose now that $u$ belongs to some slice $C_{\xi, \xi'}$. Let $\Xi_u$ be the set of $\xi$ such that $C_\xi$ contains $u$ and let $\xi \in \Xi_u$. A polygonal neighbourhood of $u=[ \xi, x, t]$ in $C_\xi \in \Xi$ is obtained as follows. Using the construction described in Definition \ref{stabinterval}, we choose a point $x$ of $T \cap C_\xi$ that sees the apex $O_\xi$ and $u$ with an unoriented angle $\theta_c$. The geodesic line between $x$ and $u$, along with it symmetric with respect to the radius $[O_\xi, u)$ define two lines meeting at $u$, which we use to define four small segments $a_\xi, b_\xi, c_\xi, d_\xi$ issuing from $u$, with an unoriented angle $\theta(t)$ and $\pi - \theta(t)$ with the ray $[O_\xi, u)$, as depicted in Figure \ref{link1} (note that we have $\theta(t) \geq \theta_c \geq \frac{\pi}{3}$). 

\begin{figure}[H]
\center
\scalebox{0.65}{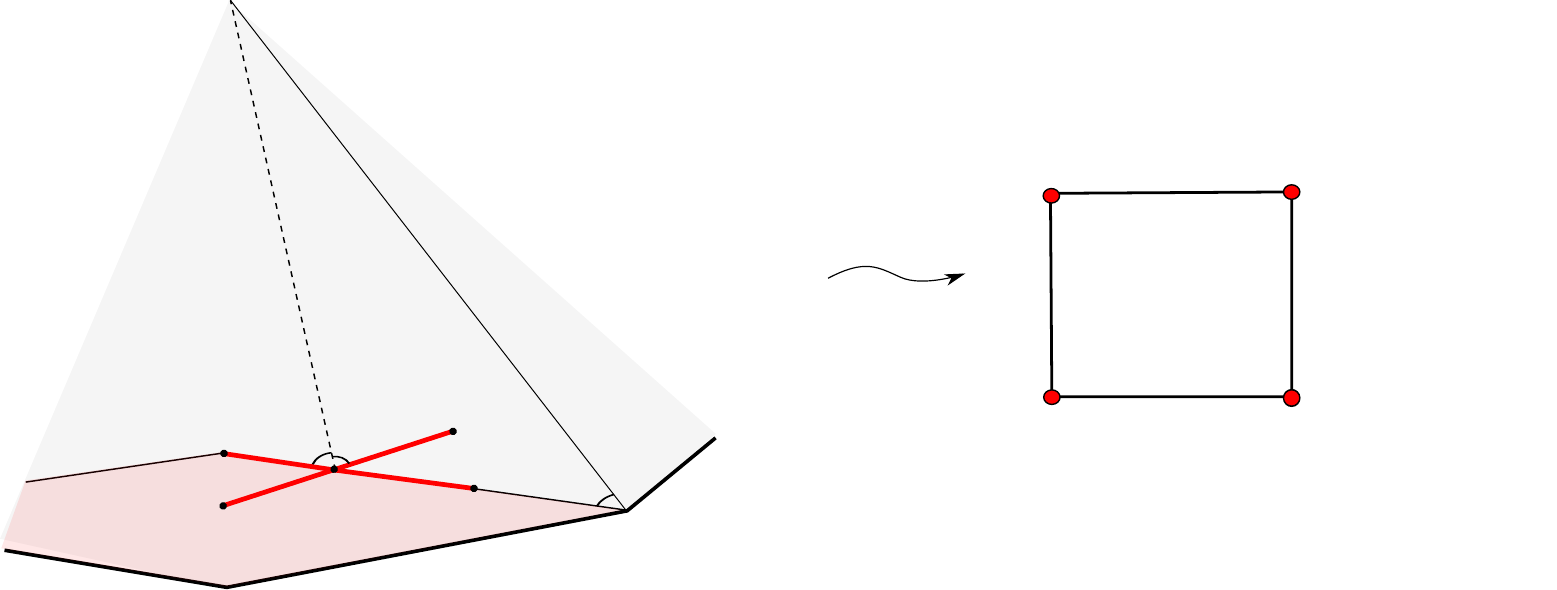}
\caption{The link $\mbox{lk}(u, C_\xi)$.}
\label{link1}
\end{figure}

We now explain how these graphs are glued together under the identifications defining $\widehat{T}_\mathrm{slices}$. Let $\xi, \xi' \in \Xi_u$ and let us look at $u$ inside $C_\xi$. If $u$ belongs to the interior of the slice $C_{\xi, \xi'}$, then the two graphs are identified in the obvious way. If $u$ belongs to the boundary of $C_{\xi, \xi'}$, then we are in the configuration of Figure \ref{link1} and the two graphs are glued as follows: 

\begin{figure}[H]
\center
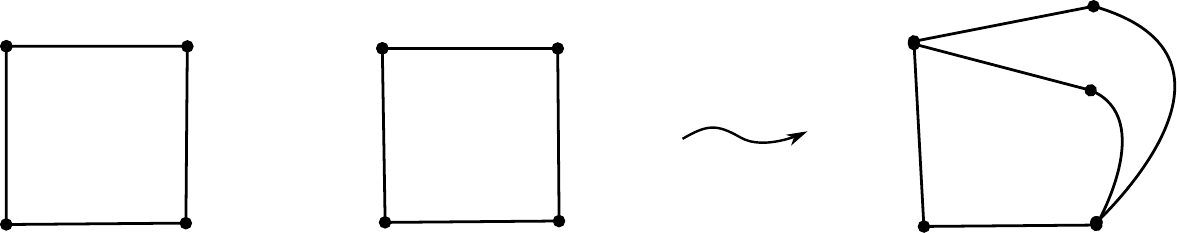
\caption{Some link identifications.}
\label{linkidentification1}
\end{figure}

Thus, the link of $u$ in $\widehat{T}_\mathrm{slices}$ is a graph without loop or double edge, and which has four types of vertices: a vertex $a$ corresponding to edges $a_\xi$ after identification, a vertex $b$ corresponding to edges $b_\xi$ after identification, vertices $c_1, c_2, \ldots$ corresponding to edges $c_\xi$, which are of valence at least $2$,
and vertices $d_1, d_2, \ldots$ corresponding to edges $d_\xi$, which are of valence at least $2$. Moreover, the following holds:
\begin{itemize}
 \item There is exactly one edge between $a$ and $b$ (of length $2\theta(t)$).
 \item There is exactly one edge between $a$ and each $d_i$ (of length $\pi - 2\theta(t)$) and  exactly one edge between $b$ and each $c_i$ (of length $\pi - 2\theta(t)$).
 \item The graph is bipartite with respect to the decomposition of the set of vertices into the sets $\{a\} \cup \{c_1, c_2, \ldots\}$ and $\{b\} \cup \{d_1, d_2 \ldots\}$. 
 \item Edges of the form $[c_i, d_j]$ are of length $2 \theta(t)$.
\end{itemize}

\begin{figure}[H]
\center
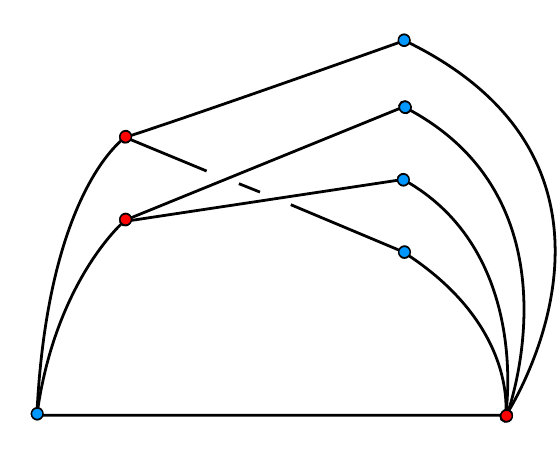
\caption{The link $\mbox{lk}(u, \widehat{T}_\mathrm{slices})$.}
\label{linkgraph1}
\end{figure}

\begin{lem}
 An injective loop in the link $\mbox{lk}(u,\widehat{T}_\mathrm{slices})$ has length at least $2\pi$.
\label{CAT(0)3}
\end{lem}

\begin{proof}
 As the link is a bipartite graph, an injective loop contains an even number of edges. Since the graph has no double edge, the loop is made of at least four edges. Now since the subgraph with all but one of the edges of the form $[a, b]$ or $[c_i, d_j]$ removed is a tree, the loop must contains two edges of the form $[a, b]$ or $[c_i, d_j]$. As these edges have length $2\theta(t) \geq \frac{2\pi}{3}$ by lemma \ref{lemmetechnique1} and the remaining ones have length $\pi-2\theta(t) \leq \frac{\pi}{3}$, such a loop has length at least $2.2\theta(t) + 2(\pi-2\theta(t)) = 2\pi$.
\end{proof}

\noindent \textbf{Points in the Bass--Serre tree.} Let $v$ be a point in the Bass--Serre tree $T$. If $v$ is not a vertex of $T$, then a neighbourhood of $v$ in $\widehat{T}_\mathrm{slices}$ is given by choosing a small neighbourhood of $v$ in any cone containing it, hence such points have CAT(0) neighbourhoods. Now let $v$ be a vertex of $T$. Let $\Xi_v$ be the set of $\xi$ such that $C_\xi$ contains $v$ and let $\xi \in \Xi_v$. A polygonal neighbourhood of $v$ in $C_\xi $ is obtained as follows. Let $a_\xi, a_\xi'$ be the two edges of $T$ issuing from $v$ that are contained in $C_\xi$. Let $c_\xi$ be the radius $[O_\xi, v]$. Let $b_\xi$ (resp. $b_\xi'$) be the segment  of $C_\xi$ issuing from $v$ that makes an unoriented angle $ \theta_c$ with the radius $c_\xi$. We use these segments to define an arbitrarily small polygonal neighbourhood of $u$ as indicated in the following picture:

\begin{figure}[H]
\center
\scalebox{0.75}{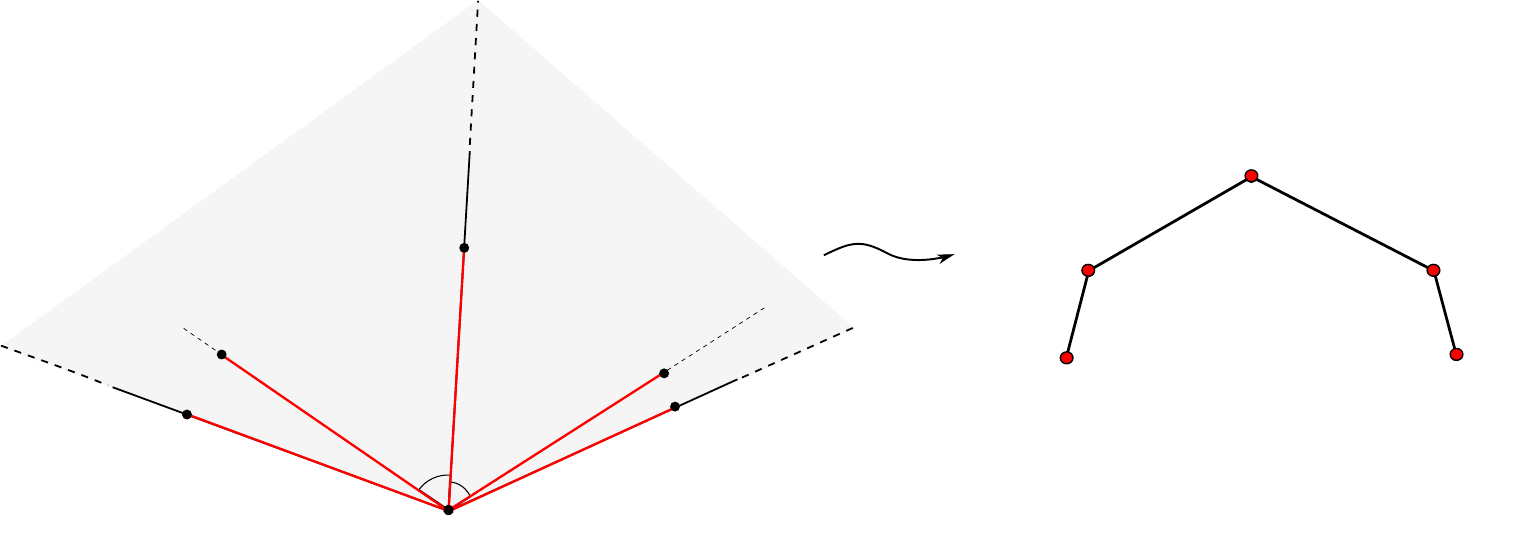}
\caption{The link $\mbox{lk}(v, C_\xi)$.}
\label{link2}
\end{figure}

We now look at how the links $\mbox{lk}(v, C_\xi)$ and $\mbox{lk}(v, C_{\xi'})$ are glued together. Let $\xi, \xi' \in \Xi_v$ and let us look at $v$ inside the cone $C_\xi$. If $v$ belongs to the interior of the slice $C_{\xi, \xi'}$, then the two links are identified in the obvious way. If $v$ belongs to the boundary of $C_{\xi, \xi'}$, then the two links are glued along a common edge as follows:

\begin{figure}[H]
\center
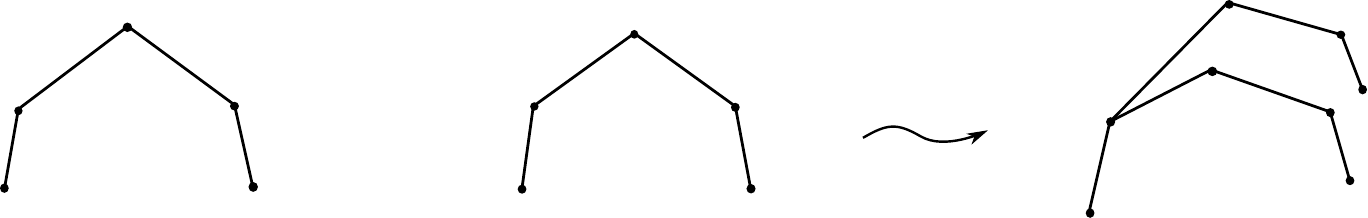
\caption{Some link identifications.}
\label{linkidentification2}
\end{figure}

Thus, the link $\mbox{lk}(v, \widehat{T}_\mathrm{slices})$ is a graph with no double edge or loop and which has three types of vertices: 
\begin{itemize}
 \item Vertices $a_1, a_2, \ldots$ (type A) corresponding to edges of $T$. These vertices are of valence $1$.
 \item Vertices $b_1, b_2, \ldots$ (type B) corresponding to segments $b_\xi, b_\xi'$, $\xi \in \Xi_v$. These vertices are of valence at least $2$.
 \item Vertices $c_1, c_2, \ldots$ (type C) corresponding to edges $c_\xi$, $\xi \in \Xi_v$. These vertices are of valence $2$. 
\end{itemize}
 Furthermore, $\mbox{lk}(v, \widehat{T}_\mathrm{slices})$ is a tripartite graph with respect to the partition of the set of its vertices into the aforementioned three types A, B and C. 

\begin{figure}[H]
\center
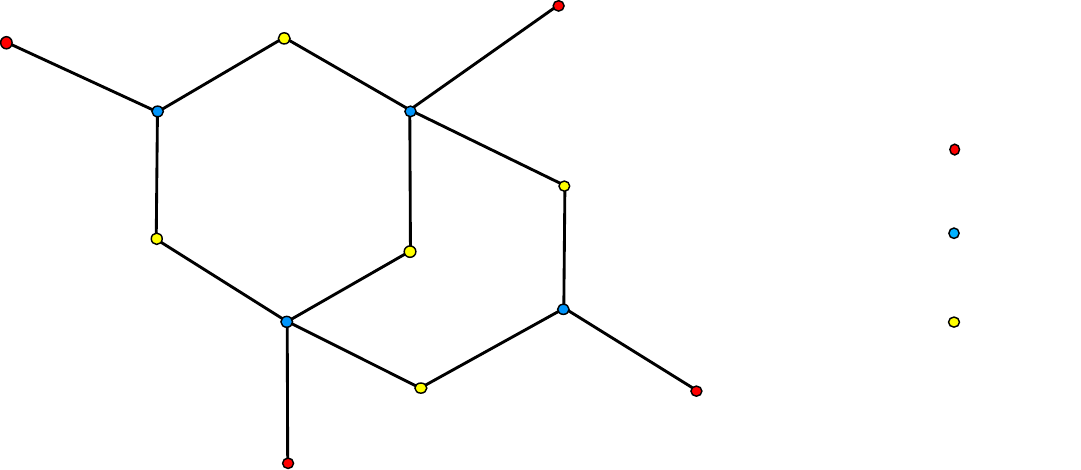
\caption{A portion of the link $\mbox{lk}(v, \widehat{T}_\mathrm{slices})$.}
\label{linkgraph2}
\end{figure}

\begin{lem}
 An injective loop in the link $\mbox{lk}(v,\widehat{T}_\mathrm{slices})$ has length at least $2\pi$.
\label{CAT(0)4}
\end{lem}

\begin{proof}
 Let $\gamma$ be an injective loop in $\mbox{lk}(v,\widehat{T}_\mathrm{slices})$. Since type A vertices have valence $1$, $\gamma$ only meets type B and type C vertices. Moreover, $\gamma$ is a bipartite graph for the induced colouring, hence it has an even number of edges. As there is no double edge, $\gamma$ has at least four edges.

We prove by contradiction that it cannot contain exactly four edges. Indeed, $\gamma$ would then contain two type C vertices corresponding to edges $c_\xi, c_{\xi'}$ ($\xi, \xi' \in \Xi_v$) and the remaining two vertices would thus correspond to the associated edges $b_\xi, b_{\xi}', b_{\xi'}, b_{\xi'}'$ after identification. Consequently, $\gamma$ would be contained in the image of $\mbox{lk}(v, C_\xi) \cup \mbox{lk}(v, C_{\xi'})$ after identification, but the above discussion shows that this image does not contain an injective cycle (see Figure \ref{linkidentification2}).

Thus, $\gamma$ contains at least six edges, all of whose being between a type B vertex and a type C vertex. As the length of such an edge is $\theta_c > \frac{\pi}{3}$ by Lemma \ref{lemmetechnique1}, the length of $\gamma$ is at least $6\theta_c >2\pi$.
\end{proof}

\begin{cor} The modified coned-off space $\widehat{T}_\mathrm{slices}$ is CAT(0). \qed
\end{cor}

\bibliographystyle{plain}
\bibliography{PetiteSimplification}

\end{document}